\documentclass[a4paper]{article}

\usepackage[cp1251]{inputenc} 
\usepackage[T2A]{fontenc} 

\usepackage[english]{babel}
\usepackage{amsmath, amsfonts, amssymb, amsthm}
\usepackage{floatflt}
\usepackage{multicol}
\usepackage{amsfonts}
\usepackage{caption}
\usepackage{xcolor}

\usepackage{pdfpages} 
\usepackage[12pt]{extsizes} 

\usepackage{graphicx} 

 \newtheorem{theorem}{Theorem}
\newtheorem*{theorem*}{Theorem 1}
\newtheorem*{theorem**}{Theorem 2}
\newtheorem{lemma}{Lemma}[section]
\newtheorem{prop}{Proposition}

\newtheorem{remark}{Remark}
\newtheorem{cor}{Corollary}[section]

\newcommand{\Sp}{\mathbb{S}}

\numberwithin{equation}{section}

\DeclareCaptionLabelSeparator{dot}{. }
\title{Simple closed geodesics on regular tetrahedra in spherical space}
 \author{Alexander ~A.~Borisenko, ~Darya ~D.~Sukhorebska\footnote{
 The second author is supported by IMU Breakout Graduate Fellowship}}

\begin{document}
\date{}
\maketitle

{\it Abstract}.
On a regular tetrahedron in spherical space
there exist the finite number of simple closed geodesics.
For any pair of coprime integers $(p,q)$ 
it was  found the numbers $\alpha_1$ and $\alpha_2$  depending on $p$, $q$
and satisfying the inequalities $\pi/3< \alpha_1 < \alpha_2 < 2\pi/3$ such that 
 on a regular tetrahedron in spherical space 
 with the faces angle  $\alpha \in \left( \pi/3, \alpha_1 \right)$ 
 there exists unique, up to the rigid motion of the tetrahedron, 
 simple closed geodesic of type $(p,q)$, and
  on a regular tetrahedron  with the faces angle $\alpha \in \left( \alpha_2, 2\pi/3 \right)$
   there is no simple closed geodesic of type $(p,q)$.

{\it Keywords}: closed geodesic, regular tetrahedron, spherical space.
  
{\it MSC}: 53С22, 52B10


\section{Introduction}\label{s1}
 
Working on the three-body problem, Poincare conjectured the existing of a simple
 (without points of self-intersection) closed geodesic  
 on a smooth closed convex surface in three-dimensional Euclidean space.
 In 1929  Lyusternik and  Shnirelman proved that
 on a  Riemannian manifold homeomorphic to a sphere there exist
 at least three simple closed geodesics~(see~\cite{LustShnir},~\cite{LustShnir47}).

In 1898 Hadamard showed that on a closed surface of negative curvature
any closed curve, that is not homotopic to zero, could be  deformed into the convex curve
of minimal length within  its free homotopy group. 
 This minimal curve is unique  and it is  a closed geodesic  (see~\cite{Hadamard}).
An interesting problem is to find the asymptotic behavior of the number of simple closed
geodesics, depending on the length of these geodesics, on a compact manifold of negative curvature
For instance,  Huber proved that on a complete closed two-dimensional manifold of constant
negative curvature the number of closed geodesics of length at most $L$ has the order of growth 
 $e^{L}/L$ as $L \rightarrow +\infty$ (see. \cite{Huber59}, \cite{Huber61}). 
 In Rivin's work~\cite{Rivin}, and later in Mirzakhani's work~\cite{Mirz08},  it's proved that 
 on a surface of a constant negative curvature of genus $g$ and with $n$ cusps  (points at infinity)
the number of simple closed geodesics of length at most $L$ 
is   asymptotic  to (positive) constant times $L^{6g-6+2n}$ as $L \rightarrow +\infty$. 

The substantive results about the behavior of geodesic lines on 
 a convex two-dimensional surface was found by Cohn-Vossen~ \cite{Kon-Fosen},
 Alexandrov~ \cite{Alek50}, Pogorelov~\cite{Pogor69}.
In one of the earliest work   Pogorelov  proved that
a geodesic of length $\le \pi / \sqrt{k}$ realized the shortest path between its endpoint 
 on a closed  convex surface of the Gaussian curvature   $\le k$~ \cite{Pog46}.
Toponogov proved that on  $C^2$-regular closed surface of curvature $\ge k >0$ the length 
of a simple closed geodesic  is $\le 2\pi / \sqrt{k}$~\cite{Toponog63}.
  Vaigant and Matukevich obtained that on this surface 
a geodesic of length   $\ge  3\pi / \sqrt{k}$ has point of self-intersection~\cite{VagMatuc}.

Geodesics have also been studied on non-smooth surfaces, including convex polyhedra
 (see~\cite{Cot05} and~\cite{Law10}).
D. Fuchs and E. Fuchs supplemented  and systematized the results on closed geodesics on
regular polyhedra in three-dimensional Euclidean space (see~\cite{FucFuc07} and~\cite{Fuc09}).
 Protasov obtained a condition for the existence of simple closed geodesics 
 on a tetrahedron in Euclidean space
  and evaluated the number of these geodesics
in terms of the difference between $\pi$
and the sum of the angles at a vertex of the tetrahedron~\cite{Pro07}.
 
A simple closed geodesic is said to be \textit{of type $(p, q)$} if
it has $p$ vertices on each of two opposite edges of the tetrahedron, $q$ vertices
on each of other two opposite edges, and $p + q$ vertices on each of the remaining
two opposite edges.
Geodesics are called  equivalent  if they intersect the same edges of the tetrahedron
in the same order.

On a regular tetrahedron in Euclidean space, for each ordered pair of
coprime integers $(p, q)$ there exists a  class of equivalent simple closed geodesics of type $(p,q)$, 
up to the isometry of the tetrahedron. 
Each of these classes contains an infinity many  geodesics. 
Furthermore, into the class there is a   simple close geodesic passing through the midpoints of two pairs
of opposite edges of the tetrahedron.

In~ \cite{BorSuh} we studied simple closed geodesics on a regular tetrahedra
 in Lobachevsky (hyperbolic) three-dimensional space. 
In Euclidean space, the faces of a tetrahedron have zero
Gaussian curvature, and the curvature of a tetrahedron is concentrated only at its
vertices. 
In Lobachevsky space, the Gaussian curvature of  faces is -1, then the
curvature of a tetrahedron is determined not only by its vertices, but also by its faces.
Moreover, in hyperbolic space   the value $\alpha$ of faces angle satisfies  $0< \alpha < \pi/3$.
The intrinsic geometry of such tetrahedron depends on the value of its faces angle.
It follows that the behavior of closed geodesics on
a regular tetrahedron in Lobachevsky space differs from the Euclidean case. 

It is proved that on a regular tetrahedron in hyperbolic space 
for any coprime integers $(p, q)$, $0 \le p<q$,  
there exists unique, up to the rigid motion of the tetrahedron, simple closed geodesic of type $(p,q)$, 
and it   passes through the midpoints of two pairs of opposite edges of the tetrahedron.
These geodesics exhaust all simple closed geodesics on a regular tetrahedron in hyperbolic space.
 The number of simple closed geodesics of length bounded by $L$ is asymptotic to constant
  (depending on $\alpha$) times $L^2$, when $L$ tends to infinity \cite{BorSuh}.

  In this work we considered simple closed geodesics on a regular
    tetrahedron in spherical three-dimensional space. 
  In this space  the  curvature of a face equals 1, then the
curvature of a tetrahedron is also determined   by its vertices and  faces.
  The intrinsic geometry of a tetrahedron depends on the value $\alpha$ of its faces angle, 
  where $\alpha$ satisfies   $\pi/3 < \alpha \le 2\pi/3$.
If $\alpha=2\pi/3$, then the tetrahedron coincides with the unit  two-dimensional sphere.
 Hence there are infinitely many simple closed geodesics on it 
 and they are great circles of the sphere. 
  
  \textit{On a regular tetrahedron in spherical space there exists the finite number of simple closed geodesics. 
The length of all these geodesics is less than $2\pi$.}
 \par  \textit{For any coprime integers $(p, q)$  we presented the numbers $\alpha_1$ and $\alpha_2$,
  depending on  $p$, $q$ and satisfying the inequalities $\pi/3< \alpha_1 < \alpha_2 < 2\pi/3$, such that \\
1) if $\pi/3< \alpha <\alpha_1$, then 
on a regular tetrahedron in spherical space with the faces angle  $\alpha$ 
 there exists unique simple closed geodesic of type $(p,q)$, up to the rigid motion of this tetrahedron; \\
2)  if  $\alpha_2 < \alpha < 2\pi/3$, then 
on a regular tetrahedron  with the faces angle   $\alpha$  there is not simple closed geodesic of type $(p,q)$.}

\section{Definitions}\label{construction}

A \textit{geodesic} is  locally the shortest curve.
On a convex polyhedron, a geodesic has the following properties (see~\cite{Alek50}):\\
1) it consists of line segments on faces of the polyhedron;\\
2) it forms equal angles with edges of adjacent faces;\\
3) a geodesic cannot pass through a vertex of a convex polyhedron.

Note, that by the `line segment’
we mean  a geodesic segment in a space of constant curvature, where the polyhedron lies in.

Let us take two tetrahedra in the spaces of constant curvature 
and consider a closed geodesic on each of them.
Construct a  bijection between the vertices of the tetrahedra
and give the same labels to the corresponding vertices.
Hence   closed geodesics on these  tetrahedra is called \textit{equivalent}
 if they intersect the same-labeling edges in the same order~ \cite{Pro07}.

Fix the point of a geodesic on a tetrahedron's edge and roll the
tetrahedron along the plane in such way that the geodesic always touches the plane.
The traces of the faces form the \textit{development} of the tetrahedron on a plane
 and the geodesic is a line segment inside the development.

A \textit{spherical triangle} is a convex polygon on a unit sphere bounded by three the shortest lines. 
A \textit{regular tetrahedron} in three-dimensional  spherical space  $\Sp^3$ is a closed convex polyhedron
such that all its faces are regular spherical triangles and all its vertices are regular trihedral angles.
The value  $\alpha$ of its faces angle satisfies the conditions  $\pi/3 < \alpha \le 2\pi/3$.
Note, than there exist a unique (up to the rigid motion) tetrahedron  in spherical space with a given value of a faces angle.
The edge length is equal 
\begin{equation}\label{a}
a =\text{arccos} \left(  \frac{\cos\alpha}{1-\cos\alpha}  \right), 
\end{equation}
\begin{equation} \label{alim}
\lim\limits_{\alpha\to\frac{\pi}{3} } a = 0; \;\;\; 
\lim\limits_{\alpha\to\frac{\pi}{2} } a = \frac{\pi}{2}; \;\;\;
\lim\limits_{\alpha\to\frac{2\pi}{3} }a = \pi - \text{arccos}\frac{1}{3}.
\end{equation}

If $\alpha=2\pi/3$, then a tetrahedron coincides with a unit  two-dimensional sphere.
 Hence there are infinitely many simple closed geodesics on it.
In the following we consider  $\alpha$ that   $\pi/3 < \alpha < 2\pi/3$.

\section{Closed geodesics on   regular tetrahedra in Euclidean space}

Consider a  regular tetrahedron $A_1A_2A_3A_4$ in Euclidean space with the edge of length $1$.
A development of the tetrahedron is a part of the standard triangulation of the Euclidean plane.
Denote the vertices of the triangulation in accordance with the vertices of the tetrahedron.
Choose two identically oriented edges   $A_1A_2$ of the triangulation, which don't belong to the same line.
 Take two points $X$ and $X'$ at equal distances from the vertex $A_1$ such that the segment  $XX'$ doesn't
 contain any vertex of the  triangulation.
Hence the segment $XX'$ corresponds to the closed geodesic on the tetrahedron $A_1A_2A_3A_4$.
 Any closed geodesic on a regular tetrahedron in Euclidean space can be constructed in this way  
 (see  Figure~\ref{evclid_case}).
 
 Note, that the segments of geodesics lying on the same face of the tetrahedron are parallel to each other.
 It follows that any  closed geodesic on a regular tetrahedron in Euclidean space
  does not have points of self-intersection.
 
 \begin{figure}[h]
\begin{center}
\includegraphics[width=140mm]{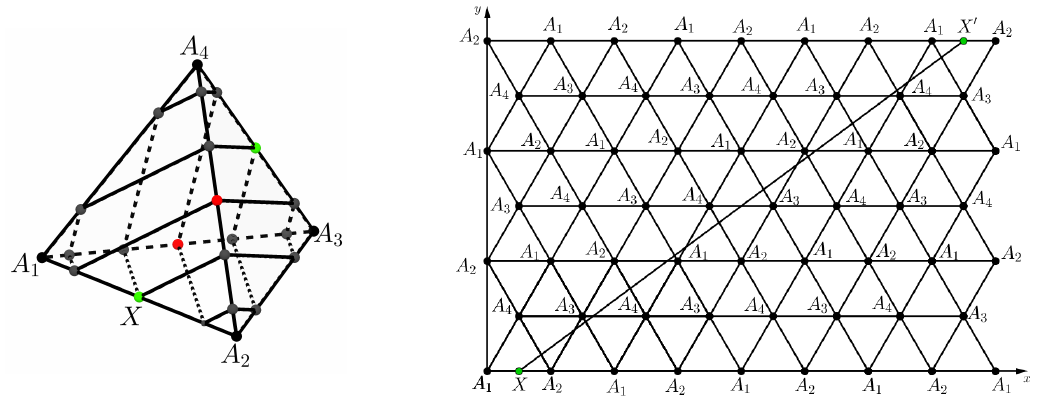}
\caption{ }
\label{evclid_case}
\end{center}
\end{figure}

We introduce a rectangular Cartesian coordinate system with the origin at $A_1$ 
and the $x$-axis along the edge  $A_1A_2$ containing $X$.
Then the vertices $A_1$ and $A_2$ has the coordinates  $\left( l, k\sqrt{3} \right)$,
and the coordinates of $A_3$ and $A_4$ are $\left( l+1/2, k\sqrt{3} +1/2 \right)$,
where  $k, l $ are integers.
The coordinates of  $X$ and $X'$ equal $(\mu, 0)$ and $(\mu+q+2p, q\sqrt3)$, where $0<\mu<1$.
The segment $XX'$ corresponds to the  simple closed geodesic  $\gamma$ of type $(p,q)$
on a regular tetrahedron in Euclidean space.
If $(p,q)$ are coprime  integers  then  $\gamma$ does not repeat itself.
The length of  $\gamma$ is equal 
\begin{equation}\label{length_evcl_geod}
L = 2 \sqrt{p^2+pq+q^2}.
\end{equation}

Note, that  for each coprime integers $(p, q)$  there exist infinitely many simple closed geodesics of type $(p, q)$,
and all of them are parallel in the development and intersect the tetrahedron's edges in the same order.

If $q=0$ and $p=1$, then geodesic  consists of four segments  
 that consecutively intersect four edges of the tetrahedron,
  and doesn't go through the one pair of  opposite edges.

\begin{prop}\label{centreqv}\textnormal{(see~\cite{BorSuh})}
For each pair of coprime integers  $(p, q)$ there exists a  simple closed geodesic 
 intersecting the midpoints of two pairs of opposite edges 
 of the regular tetrahedron in Euclidean space.
 \end{prop}

\begin{prop}\label{corparts2}\textnormal{(see~\cite{BorSuh}) }
The development of the tetrahedron obtained by unrolling along
a closed geodesic consists of four equal polygons, and any two adjacent polygons can
be transformed into each other by a rotation through an angle $\pi$  around the midpoint
of their common edge.
\end{prop}

 \begin{lemma}\label{evcl_dist_vertex_lem}
 Let $\gamma$ be a simple closed geodesic of type  $(p,q)$ on a regular tetrahedron in Euclidean space such that
 $\gamma$  intersects the midpoints of two pairs of opposite edges.
 Then the distance $h$ from the tetrahedron's vertices to $\gamma$ satisfies the inequality 
\begin{equation}\label{evcl_dist_vertex}
h \ge \frac{ \sqrt{3} }{ 4 \sqrt{p^2+pq+q^2} } . 
\end{equation}
\end{lemma}

\begin{proof}
Let us take  a  regular tetrahedron $A_1A_2A_3A_4$ in Euclidean space with the   edge of length~$1$.
Suppose  $\gamma$ intersects the edge $A_1A_2$ at the midpoint $X$.
Consider the development of the tetrahedron along  $\gamma$  starting from the point  $X$ 
and introduce a  Cartesian coordinate system as described above.
The geodesic $\gamma$ is unrolled into the segment $XX'$ lying at the line  $y=\frac{q\sqrt3}{q+2p} (x-\frac{1}{2})$ 
 (see Figure \ref{evclid_case}).
The segment  $XX'$ intersect the edges  $A_1A_2$ at the points 
 $(x_b,y_b)=$$\left(  \frac{ 2(q+2p) k + q}{2q},k \sqrt{3} \right)$, where $k \le q$.
 Since $\gamma$ does not pass through a vertex of the tilling, then $x_b$ couldn't be an integer.
 Hence on the edge  $A_1A_2$ the distance from the vertices to the points of $\gamma$ is not less than $1/2q$.
 
In the same way on the edge  $A_3A_2$ the distance from the vertices to the points of
 $\gamma$ is not less than $1/2p$.
 
 \begin{figure}[h]
\begin{center}
\includegraphics[width=65mm]{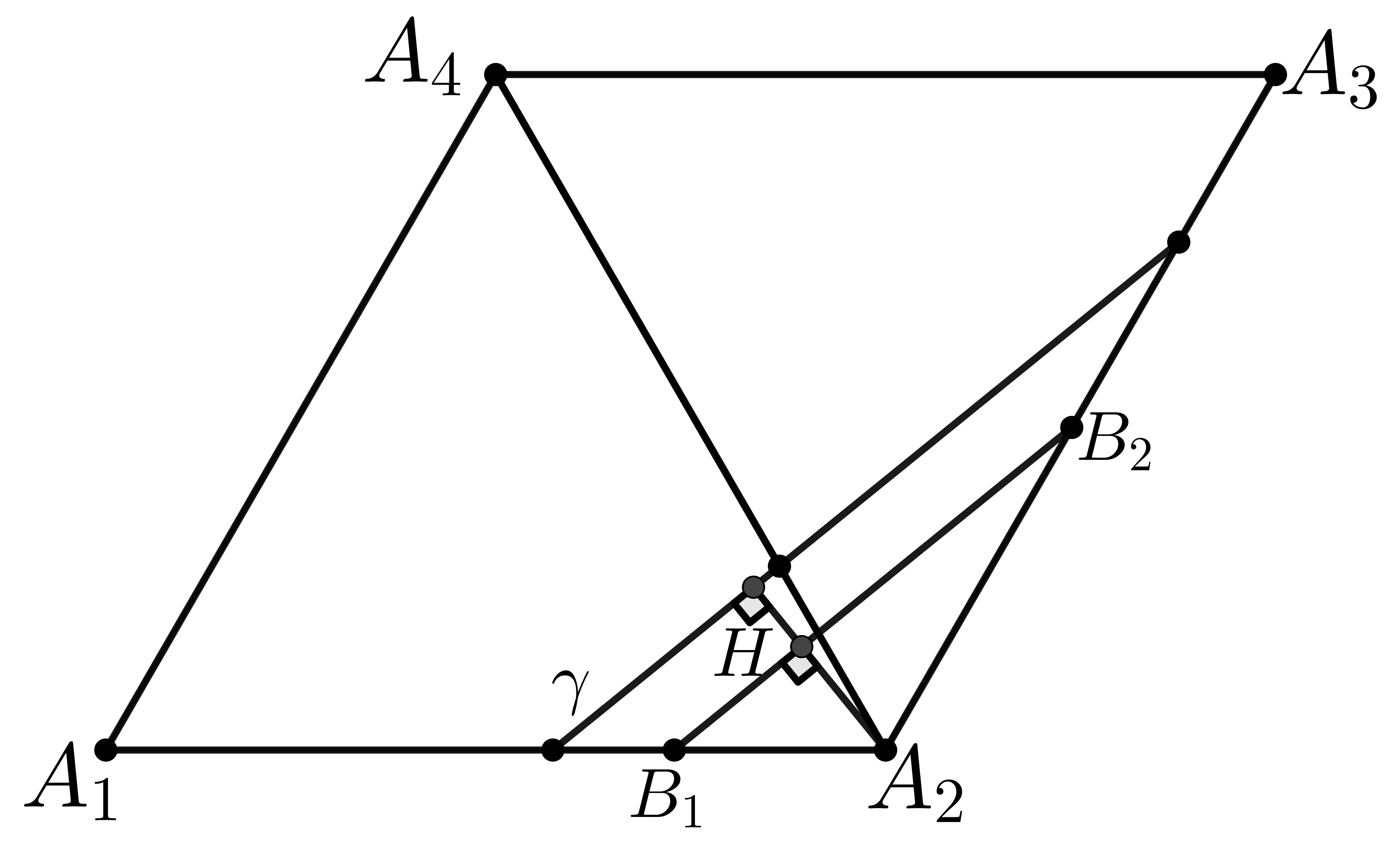}
\caption{ }
\label{A2B1B3}
\end{center}
\end{figure}

Choose the points $B_1$ at the edge $A_2A_1$ and $B_2$ at the edge $A_2A_3$ such that
 the length $A_2B_1$ is $1/2q$ and  the length $A_2B_2$ equals $1/2p$ (see Figure~\ref{A2B1B3}). 
The distance $h$ from the vertex $A_2$ to $\gamma$ is not less than the  height $A_2H$ of the triangle $B_1A_2B_2$.
The length of $B_1B_2$ equals $\frac{ \sqrt{p^2+pq+q^2}}{2pq}$.
Then the length of $A_2H$ is
\begin{equation}\label{C1B12}
|A_2H| = 
\frac{ \sqrt{3} }{ 4 \sqrt{p^2+pq+q^2} } . \notag
\end{equation}
Hence the inequality (\ref{evcl_dist_vertex}) is proved.
\end{proof}

Introduce some definitions following \cite{Pro07}.
A {\it broken line} on a tetrahedron is a curve 
 consisting of the   line segments, which connect points on the edges of this tetrahedron consecutively.
A {\it  generalized  geodesic} on a tetrahedron is a closed broken line with following properties:\\
(1)  it does not have points of self-intersection and 
adjacent segments of it lie on different faces; \\
(2) it crosses more than three edges on the tetrahedron and doesn't pass through tetrahedron's vertices.  
 
\begin{prop}\label{allgeod}
\textnormal{(V. Protasov~\cite{Pro07})}
For every  generalized  geodesic  on a tetrahedron  in   Euclidean space 
there exists a simple closed geodesic on a regular tetrahedron 
 in Euclidean space that is equivalent to this  generalized  geodesic.
\end{prop}

\section{Geodesics of type $(0,1)$ and $(1,1)$ on a  regular tetrahedron in $\Sp^3$}\label{exampl}

Let us remind  that a simple closed geodesic  $\gamma$ has  type $(p, q)$ if
it  has $p$ vertices on each of two opposite edges of the tetrahedron, $q$ vertices
on each of other two opposite edges, and $p + q$ vertices on each of the remaining
two opposite edges.
If $q=0$ and $p=1$, then geodesic  consists of four segments  
 that consecutively intersect   four edges of the tetrahedron,
  and doesn't go through the one pair of   opposite edges.

\begin{lemma}\label{geod01}
On a regular tetrahedron in spherical space there exist three different simple closed geodesics of type $(0,1)$.
They coincide under isometries of the tetrahedron. 
\end{lemma}

\begin{proof}
Consider a regular tetrahedron $A_1 A_2 A_3 A_4$ in $\Sp^3$ with the faces angle 
 $\alpha$ where $ \pi/3 <  \alpha < 2\pi/3 $. 
 Let $X_1$ and $X_2$ be the midpoints of  $A_1A_4$ and $A_3A_2$,
and $Y_1$, $Y_2$  be the midpoints of  $A_4A_2$ and $A_1A_3$ respectively.
Join these points   consecutively with the segments trough the faces. 
We obtain a closed broken line  $X_1Y_1X_2Y_2$  on the tetrahedron. 
Since the points  $X_1$, $Y_1$, $X_2$ and $Y_2$ are midpoints, then 
the triangles  $X_1A_4Y_1$, $Y_1A_2X_2$, $X_2A_3Y_2$ and $Y_2A_1X_1$ are equal.
It  follows that the broken line  $X_1Y_1X_2Y_2$  is a simple closed geodesic of type  $(0,1)$
on a regular tetrahedron in spherical space (see Figure~\ref{01geod}).
Choosing the midpoints of other pairs of opposite edges, we can construct other two geodesics of type $(0,1)$
on the tetrahedron.
\end{proof}
 \begin{figure}[h]
\begin{center}
\includegraphics[width=100mm]{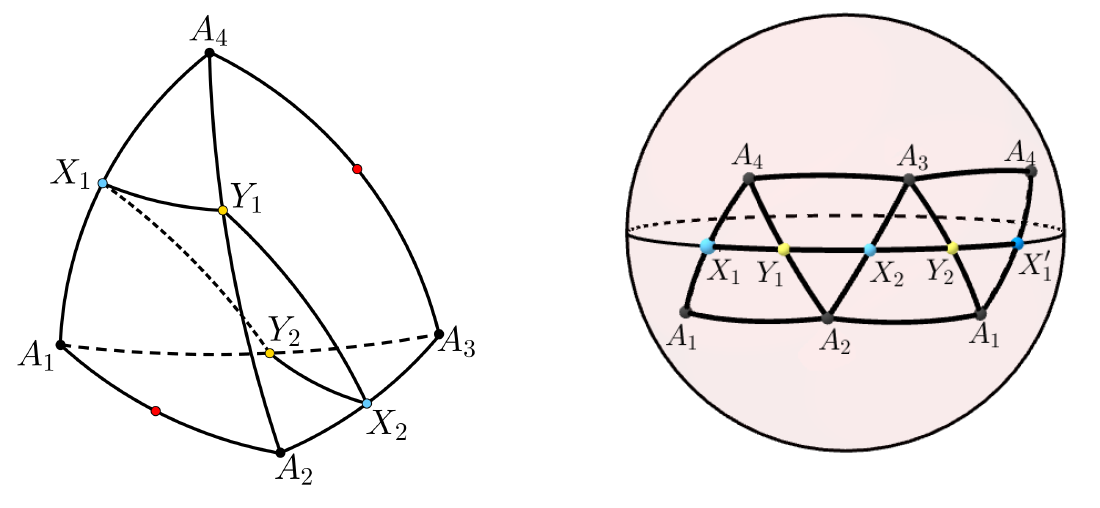}
\end{center}
\caption{Simple closed geodesic of type $(0,1)$ on a regular tetrahedron in  $\Sp^3$ }
\label{01geod}
\end{figure}

\begin{lemma}\label{geod11}
On a regular tetrahedron in spherical space with the faces angle   $\alpha < \pi/2$ 
there exist three simple closed geodesics of type $(1, 1)$.
\end{lemma}

\begin{proof}
Consider a regular tetrahedron $A_1 A_2 A_3 A_4$ in $\Sp^3$ with the faces angle  
 $\alpha$ where  $ \pi/3 <  \alpha < \pi/2 $. 
As above, the points $X_1$ and $X_2$ are the  midpoints of  $A_1A_4$ and $A_3A_2$,
and $Y_1$, $Y_2$ are the midpoints of  $A_4A_2$ and $A_1A_3$.

Develop two adjacent faces  $A_1A_4A_3$ and $A_4A_3A_2$ into the plain and 
join the points $X_1$ and $Y_1$ with the line segment.
Since $\alpha < \pi/2$, then the segment  $X_1Y_1$ is contained inside the development
and intersects the edge  $A_4A_2$  at right angle.
Then develop another  two adjacent faces $A_4A_1A_2$ and $A_1A_2A_3$ 
and construct the segment $Y_1X_2$.
In the same way join the points  $X_2$ and $Y_2$ within the faces  $A_2A_3A_4$ and  $A_3A_4A_1$,
and join    $Y_2$ and $X_1$ within $A_1A_2A_3$ and $A_4A_1A_2$  (see Figure \ref{11geod}).
Since  the points  $X_1$, $Y_1$, $X_2$ and $Y_2$ are the midpoints, it follows, that the triangles 
$X_1A_4Y_1$, $Y_1A_2X_2$, $X_2A_3Y_2$ и $Y_2A_1X_1$  are equal.
Hence, the segments $X_1Y_1$, $Y_1X_2$, $X_2Y_2$, $Y_2X_1$ form a simple closed geodesic of type $(1,1)$
on the tetrahedron.

Another two geodesics of type $(1,1)$ on a tetrahedron could be constructed in the same way, 
if we choose the midpoints of other pairs of opposite edges.
\end{proof}

 \begin{figure}[h]
\begin{center}
\includegraphics[width=105mm]{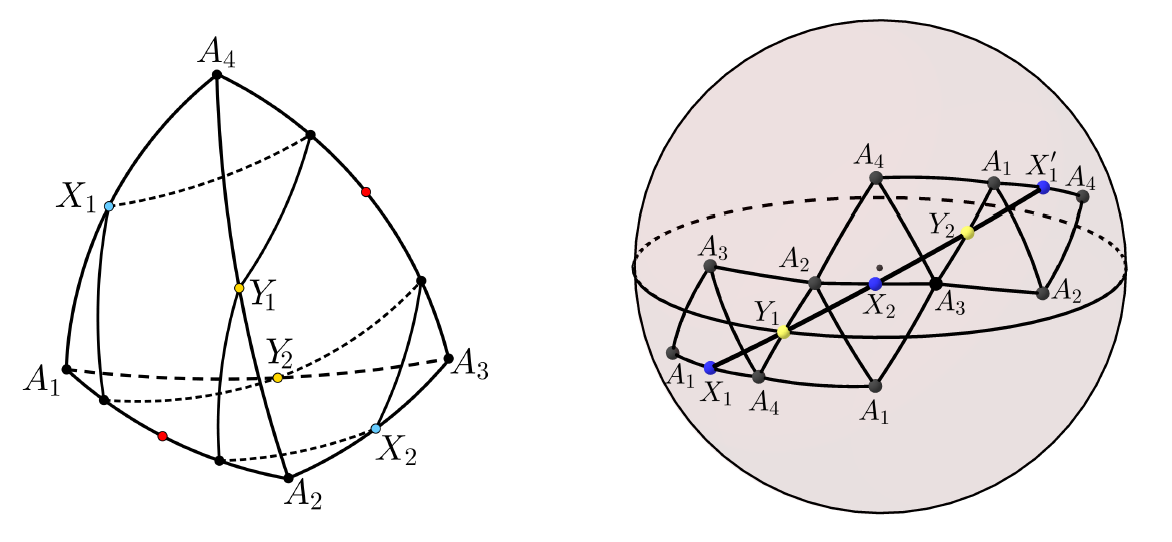}
\end{center}
\caption{Simple closed geodesic of type $(1,1)$ on a regular tetrahedron in  $\Sp^3$ }
\label{11geod}
\end{figure}
 
\begin{lemma}\label{morePi2}
On a regular tetrahedron  in spherical space  with the faces angle   $\alpha \ge \pi/2$ 
there exists only three simple closed geodesics and all of them have  type $(0,1)$.
\end{lemma}

\begin{proof}
Consider a regular  tetrahedron  in spherical space  with the faces angle    $\alpha \ge \pi/2$.
Since a geodesic is a line segment inside the development of the tetrahedron, 
then it cannot intersect three edges of the tetrahedron, coming out from the same vertex,  in succession. 

If a simple closed geodesic on the tetrahedron is of type $(p,q)$, where  $p=q=1$ or $1 < p<q$, 
then this geodesic intersect  three tetrahedron's edges, starting at the same vertex, in succession  (see~\cite{Pro07}). 
And only a simple closed geodesic of  type  $(0,1)$  intersects 
 two tetrahedron's edges, that have  a common vertex,
and doesn't  intersects  the third.
It follows that  on a regular tetrahedron in spherical space 
 with the faces angle  $ \alpha \in \left[ \pi/2, 2\pi/3 \right) $
 there exist only three simple closed geodesic of type  $(0,1)$ and others don't exist.
\end{proof}

In the next sections we will assume that    $\alpha$ satisfying   $\pi/3 < \alpha < \pi/2$.

\section{The length of a simple closed geodesic on a regular tetrahedron in $\Sp^3$}\label{seclength}

 \begin{lemma}\label{length}
 The length of a simple closed geodesic on a regular tetrahedron in spherical space is less than $2\pi$.
\end{lemma}
\begin{proof}
Consider a regular tetrahedron $A_1 A_2 A_3 A_4$ in $\Sp^3$ with the faces angle of value 
 $\alpha$, where  $ \pi/3 <  \alpha < \pi/2 $. 
A spherical space  $\Sp^3$ of curvature $1$ is realized as a unite tree-dimensional sphere 
in four-dimensional Euclidean space. 
Hence the tetrahedron  $A_1 A_2 A_3 A_4$ is situated in an open hemisphere.
Consider a tangent Euclidean space to this hemisphere at the center of the tetrahedron
(by `center of the tetrahedron' we mean a center of circumscribed sphere of the tetrahedron).
A central projection of the hemisphere to this tangent space maps the regular tetrahedron in spherical space to the 
regular tetrahedron in Euclidean space.
A simple closed geodesic  $\gamma$ on  $A_1A_2A_3A_4$ is mapped into  generalized  geodesic
on a regular Euclidean tetrahedron.
From   Proposition~\ref{allgeod} we know that  there exists a simple closed geodesic on a regular tetrahedron 
in  Euclidean space equivalent to this  generalized  geodesic.
It follows, that 
\textit{ a simple closed geodesic on a regular tetrahedron in $\Sp^3$ is also uniquely determined
 with the pair of coprime integers $(p,q)$ and has the same structure as a closed geodesic on a regular tetrahedron 
 in Euclidean space.}
 Investigate this structure following~\cite{Pro07}.
 
A vertex of a geodesic is called a \textit{catching node}  if it and two
adjacent vertices lie on the three edges coming out from the same   vertex $A_i$ of the
tetrahedron and are the  vertices of the  geodesic  nearest to  $A_i$ on these edges.
By the 'vertex of a geodesic' we mean a point of geodesic on an edge. 

\begin{prop}\label{gengeodstr} 
\textnormal{(V. Protasov~\cite{Pro07})}
Let $\gamma^1_1$ and $\gamma^2_1$ is a segments of a simple closed geodesic  $\gamma$,
 starting at a catching node on a regular tetrahedron,  
 $\gamma^1_2$ and $\gamma^2_2$ is the segments following  $\gamma^1_1$ and $\gamma^2_1$ and so on.
 Then for each  $i=2, \dots, 2p+2q-1$ the segments $\gamma^1_i$ and $\gamma^2_i$ 
 lie on the same face of the tetrahedron, and there are no other geodesic points between them. 
 The segments  $\gamma^1_{2p+2q}$ и $\gamma^2_{2p+2q}$ meet at the
second catching node of the geodesic.
\end{prop}
 \begin{figure}[h]
\begin{center}
\includegraphics[width=50mm]{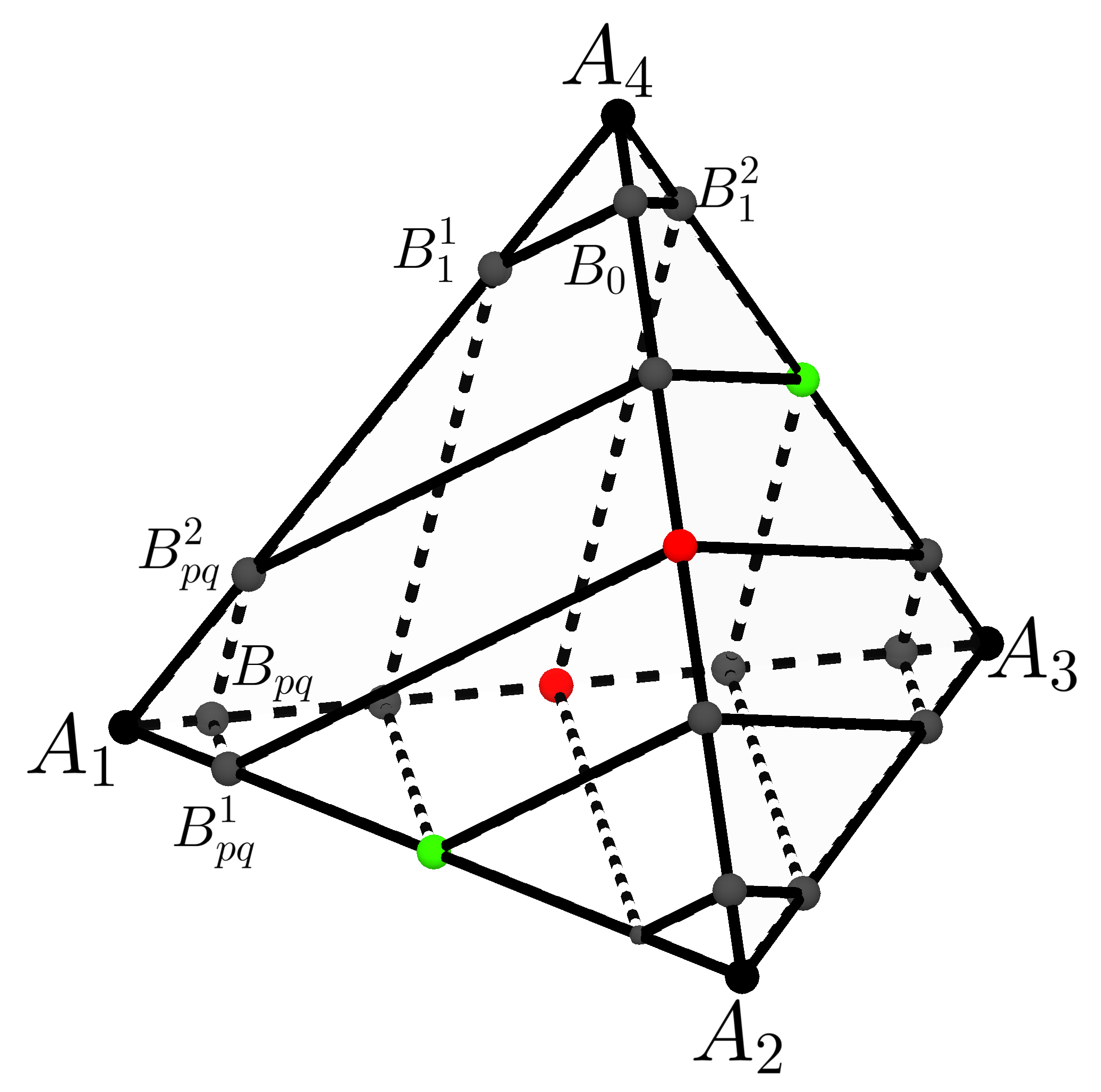}
\caption{ }
\label{geod(2,3)_zacep}
\end{center}
\end{figure}

Suppose  $\gamma$  has $q$  points on the edges $A_1A_2$ and $A_3A_4$,
$p $ points on $A_1A_4$ and $A_2A_3$, and $p+q$ points on $A_2A_4$ and $A_1A_3$.
Consider the catching node $B_0$ of  $\gamma$ at the edge $A_4A_2$.
The adjacent geodesic vertices $B^1_1$ and $ B^2_1$ are the nearest point of  $\gamma$ to $A_4$ 
at the edges $A_4A_1$ and $A_4A_3$ respectively.
The geodesic segments $B_0B^1_1$ and $B_0B^2_1$ correspond to  $\gamma^1_1$ and $\gamma^2_1$
(see Figure~\ref{geod(2,3)_zacep}). 
If we develop the faces  $A_1A_2A_4$ and $A_2A_4A_3$ into the plain,  then 
the segments   $\gamma^1_1$ and $\gamma^2_1$ will form one line segment. 
The triangle  $B^1_1A_4B^2_1$ at the development is called \textit{catching triangle} 
(see Figure~\ref{catching_triangle}). 

The segments  $\gamma^1_{2p+2q}$ and $\gamma^2_{2p+2q}$   meet at the 
second catching node  $B_{pq}$ of the geodesic.
We will assume, that  $B_{pq}$  is the nearest to $A_1$  geodesic vertex at the edge  $A_1A_3$.
The adjacent geodesic vertices  $B^1_{pq}$ and $B^2_{pq}$ are the nearest point of  $\gamma$ to $A_1$ at 
$A_1A_2$ and $A_1A_4$ respectively. 
They form the second catching triangle $B^1_{pq} A_1 B^2_{pq}$.

From these catching triangles $B^1_1A_4B^2_1$ and $B^1_{pq} A_1 B^2_{pq}$ it follows next inequalities
\begin{equation}\label{triangl_zacep}
|B^1_1 B^2_1|< |B^1_1A_4|+|A_4B^2_1|;  \;\;  |B^2_{pq}B^1_{pq}|< |B^2_{pq}A_1|+|A_1B^1_{pq}|.  
\end{equation}

Now let us develop the tetrahedron into two-dimensional sphere, 
starting from the face $A_1A_4A_3$ and go along the segments $\gamma^1_i$ and $\gamma^2_i$, 
$i=2, \dots, 2p+2q-1$.
The segments $\gamma^1_2$ and $\gamma^2_2$ start at the points   $B^1_1$ and $B^2_1$ respectively, then they
  intersect the edge  $A_1A_3$. 
  Other segments $\gamma^1_i$ and $\gamma^2_i$, $i=3, \dots, 2p+2q-2$, are unrolled into two line segments, 
  that intersect the same edges in the same order and there are no other geodesic points between them. 
  The last segments   $\gamma^1_{2p+2q-1}$ and $\gamma^2_{2p+2q-1}$ intersect the edge  $A_2A_4$, 
  then go within $A_1A_2A_4$ and end at the points  $B^1_{pq}$ and $B^2_{pq}$ respectively
   (see Figure~\ref{convex_hexagon}). 
It follows, that the tetrahedron's vertices $A_4$ and $A_1$ lie inside the spherical lune,
 formed by  two great arcs containing   $B^1_1B^1_{pq}$ and $B^2_1B^2_{pq}$.
We obtained a convex  hexagon $B^1_1A_4B^2_1B^2_{pq}A_1B^1_{pq}$ at the sphere.

\begin{figure}[h]
\columnsep=5pt
\begin{multicols}{2}
\centering{\includegraphics[width=60mm]{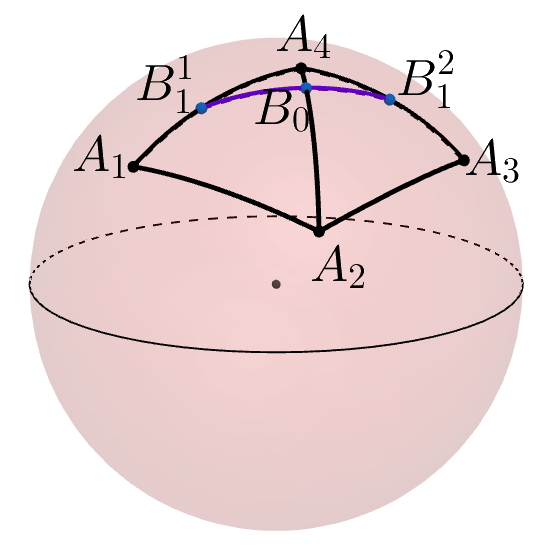}  }
\caption{ }
\label{catching_triangle}
\centering{\includegraphics[width=60mm]{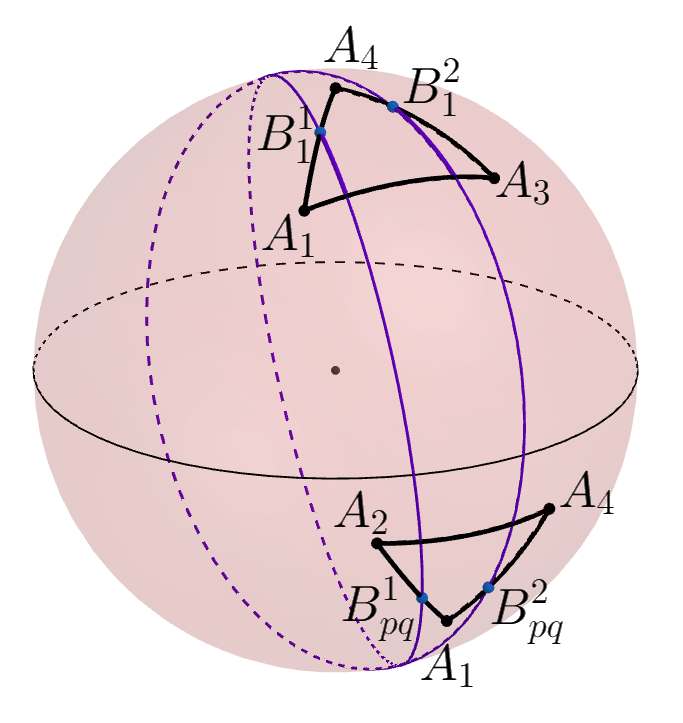} }
\caption{ }
\label{convex_hexagon}
\end{multicols}
\end{figure}

From the inequalities   (\ref{triangl_zacep}) it follows that the length of the geodesic  $\gamma$ 
is less then the perimeter of the hexagon $B^1_1A_4B^2_1B^2_{pq}A_1B^1_{pq}$.
Since the perimeter of convex polygon on a unit sphere is less  than $2\pi$,  
we get, that on a regular tetrahedron in Euclidean space 
with the face's angle  $ \alpha< \pi/2 $ the length of a simple closed geodesic 
is less than $2\pi$.

From Lemma~\ref{morePi2} we know, that if the faces angle $\alpha$ satisfies the condition $\pi/2 \le \alpha< 2\pi/3$,
then on a regular tetrahedron there exist only three simple closed geodesics and they have  type $(0,1)$. 
The length of these geodesics equals
\begin{equation}\label{length_01}
L_{0,1} = 4\arccos \left(  \frac{ \sin \frac{3\alpha}{2} }{2 \sin\frac{\alpha}{2}  }\right).  
\end{equation}
It is easy to check, that $L_{0,1} < 2\pi$, when  $\pi/2 \le \alpha< 2\pi/3$.
\end{proof}

\begin{remark} Lemma \ref{length}  could be considered as the particular case of the result~\cite{Bor2020}
 proved by first author  about the generalization of V. Toponogov theorem~\cite{Toponog63} 
  to the case of two-dimensional Alexandrov space.
\end{remark}

\section{Uniqueness of a simple closed geodesic on a regular tetrahedron in  $\Sp^3$}

In a spherical space the analogical  to Proposition~\ref{centreqv} result is true. 
\begin{lemma}\label{middle}
On a regular tetrahedron in a spherical space a simple closed geodesic 
 intersects  midpoints of two pairs of opposite edges.
\end{lemma}

\begin{proof}
Let $\gamma$ be a simple closed geodesic on a regular tetrahedron  $A_1 A_2 A_3 A_4$ in $\Sp^3$.
In the part  \ref{seclength} we showed, that  there exists a simple closed geodesic  $\widetilde{ \gamma}$ 
on a regular tetrahedron in Euclidean space such that $\widetilde{ \gamma}$  is equivalent to  $\gamma$.
From   Proposition~\ref{centreqv} we  assume  $\widetilde{ \gamma}$ intersects
the midpoints  $\widetilde{X}_1$ and $\widetilde{X}_2$ of the edges $A_1A_2$ and $A_3A_4$  on the tetrahedron.
Denote by  $X_1$ and $X_2$ the vertices of $\gamma$  at the edges $A_1A_2$ and $A_3A_4$ on a spherical tetrahedron
such that $X_1$ and $X_2$ are equivalent to the points  $\widetilde{X}_1$ and $\widetilde{X}_2$.

Consider the development of the tetrahedron along $\gamma$ starting from the point  $X_1$
 on a two-dimensional unite sphere.
 The geodesic  $\gamma$ is unrolled into the line segment  $X_1X'_1$ of length less than $2\pi$  inside the development.
 Denote by $T_1$ and $T_2$ the parts of the development along $X_1X_2$ and $X_2X'_1$ respectively.

On the tetrahedron  $A_1 A_2 A_3 A_4$ mark the midpoints  $M_1$ and $M_2$
of the edges $A_1A_2$ and $A_3A_4$ respectively.
Rotation by the angle $\pi$ over the line $M_1M_2$ is an isometry of the tetrahedron.
It follows that the development of the tetrahedron is centrally symmetric with the center $M_2$.

On the other hand symmetry over $M_2$ replaces $T_1$ and  $T_2$. 
The point  $X'_1$ at the edge $A_1A_2$ of the part  $T_2$ is mapped into the point 
  $\widehat{ X}'_1$  at the edge $A_2A_1$  containing $X_1$ on  $T_1$, and the lengths of
  $A_2X_1$ and $ \widehat{ X}'_1A_1$ are equal. 

The image of the point  $X_1$ on $T_1$ is a point $\widehat{ X}_1$ at the edge $A_1A_2$ on $T_2$.
Since $M_2$ is a midpoint of  $A_3A_4$, then the symmetry maps the point $X_2$ at  $A_3A_4$ to 
the point $\widehat{X}_2$ at the same edge  $A_3A_4$ 
such that the lengths of $A_4X_2$ and $\widehat{X}_2A_3$ are equal.
Thus, the segment $X_1X'_1$ is mapped into the segment $\widehat{ X}'_1\widehat{ X}_1$ inside the development.

\begin{figure}[h]
\begin{multicols}{2}
\hfill
\includegraphics[width=80mm]{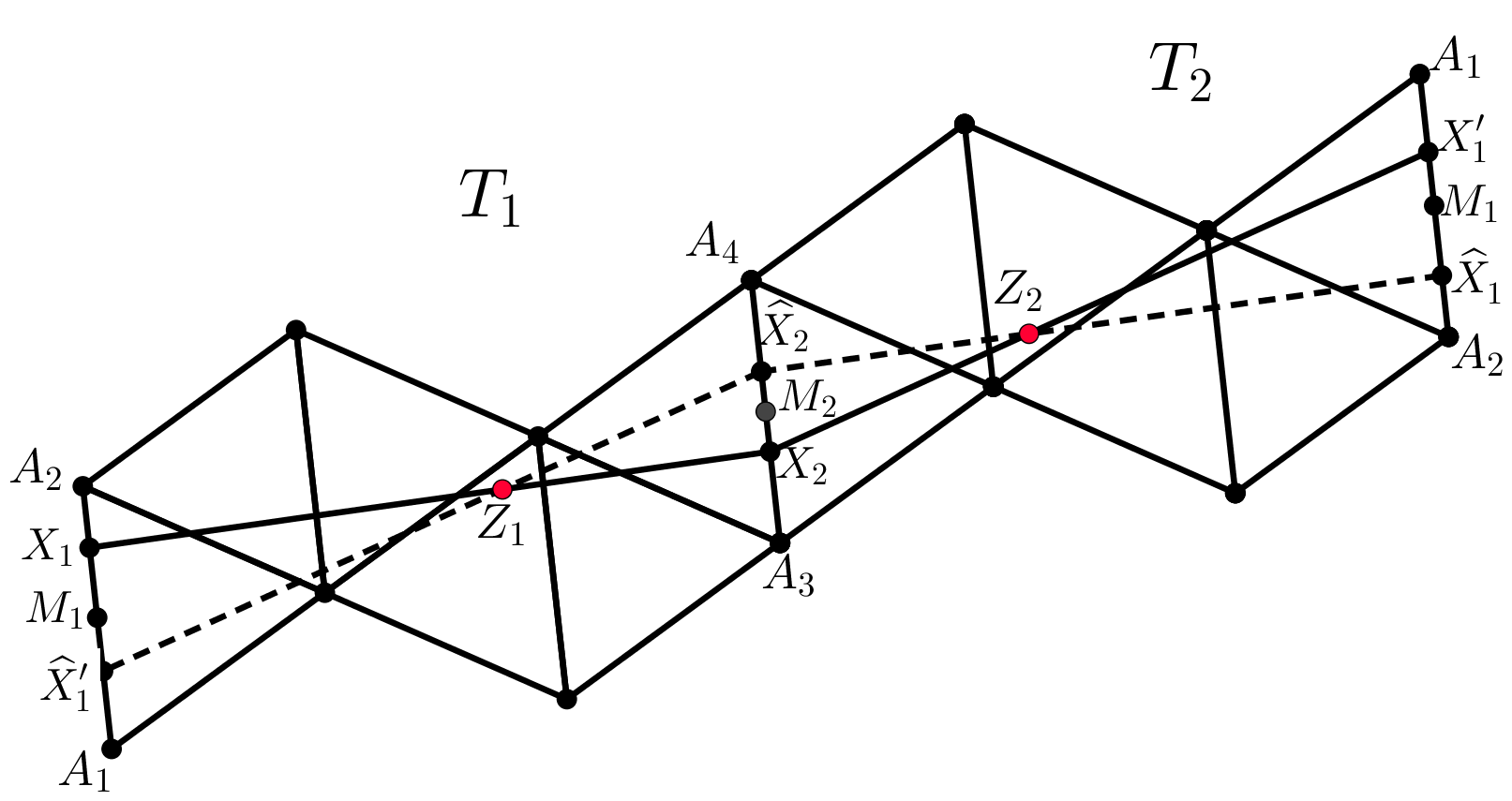}
\caption{ }
\label{uniqueness_p1}
\hfill
\includegraphics[width=80mm]{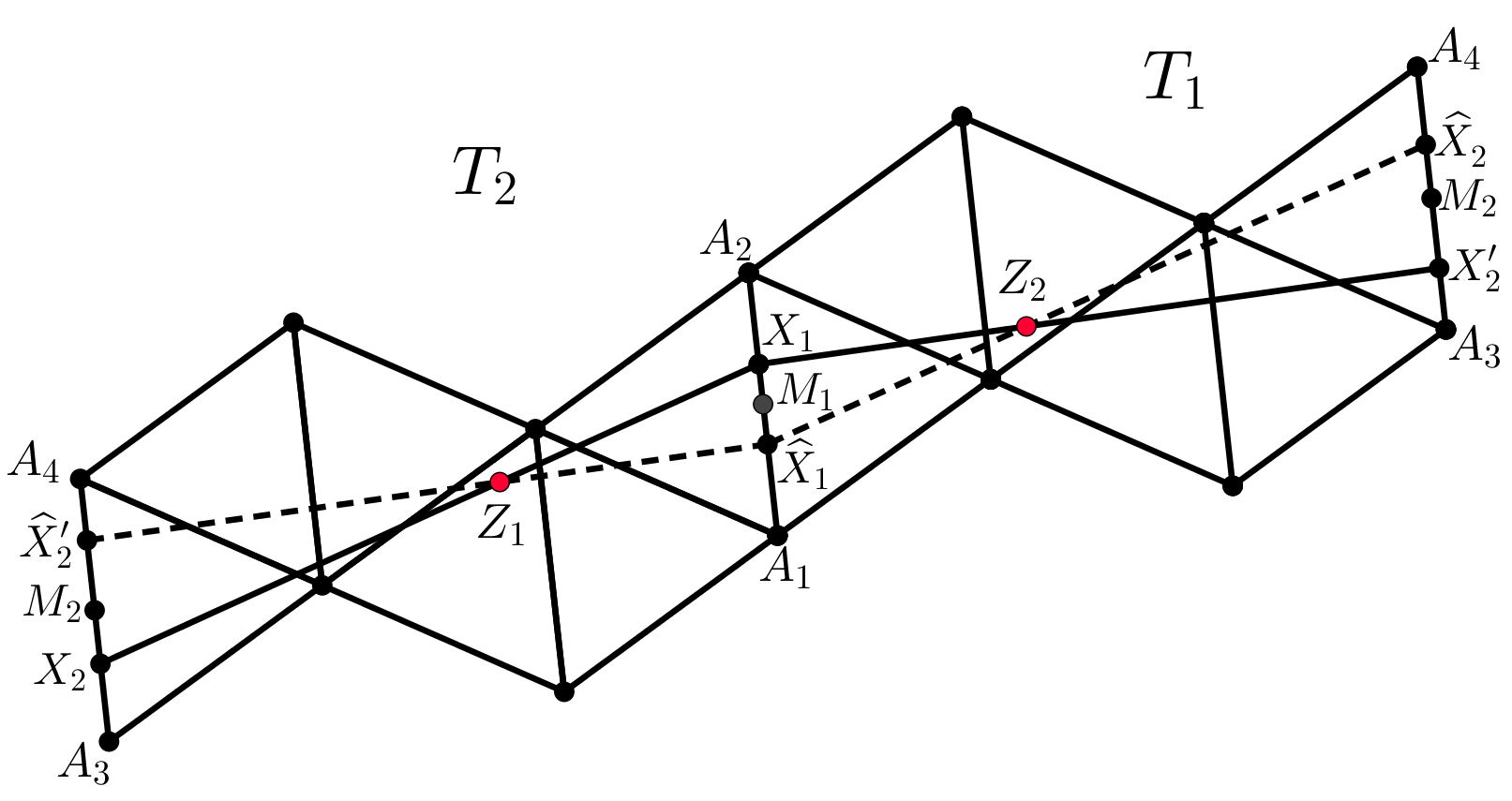}
\caption{ }
\label{uniqueness_p2}
\end{multicols}
\end{figure}

Suppose the segments  $\widehat{X}'_1\widehat{X}_2$  and $X_1X_2$ intersect at the point $Z_1$ inside $T_1$.
Then the segments $\widehat{ X}_2\widehat{ X}_1$ and $X_2X'_1$ intersect at the point $Z_2$ inside $T_2$, 
and the point  $Z_2$  is symmetric to $Z_1$ over $M_2$  (see Figure~\ref{uniqueness_p1}).
We obtain two great arcs $X_1X'_1$ and $\widehat{ X}'_1\widehat{ X}_1$ intersecting in two points.
It follows that $Z_1$ and $Z_2$ are antipodal points on a sphere and
 the length of the geodesic segment $Z_1X_2Z_2$ equals $\pi$.

Now consider the development of the tetrahedron along $\gamma$ starting from the point  $X_2$. 
This development   also consists of spherical polygons $T_2$ and $T_1$, 
but in this case they are glued by the edge $A_1A_2$ and centrally  symmetric with the center $M_1$ 
(see Figure \ref{uniqueness_p2}).

By analogy with previous case apply the symmetry over $M_1$. 
The segment  $X_2X_1X'_2$ is mapped into the segment $\widehat{ X}_2\widehat{ X}_1\widehat{ X}'_2$
inside the development.
Since the symmetries over  $M_1$ and over $M_2$ correspond to the same isometry of the tetrahedron,
then  the arcs $X_2X_1X'_2$ and  $\widehat { X}_2\widehat { X}_1\widehat { X}'_2$ also intersect 
at the points  $Z_1$ and $Z_2$ (see Figure~\ref{uniqueness_p2}).
It follows that the length of geodesic segment  $Z_1X_1Z_2$ is also equal $\pi$.
Hence the length of the geodesic $\gamma$ on a regular tetrahedron in spherical space equals  $2\pi$,
that  contradicts to   Lemma~\ref{length}.
We get that the segments $\widehat{ X}'_1\widehat{ X}_2$ and $X_1X_2$ either don't intersect or coincide
at the development part $T_1$.

 \begin{figure}[h!]
\begin{center}
\includegraphics[width=90mm]{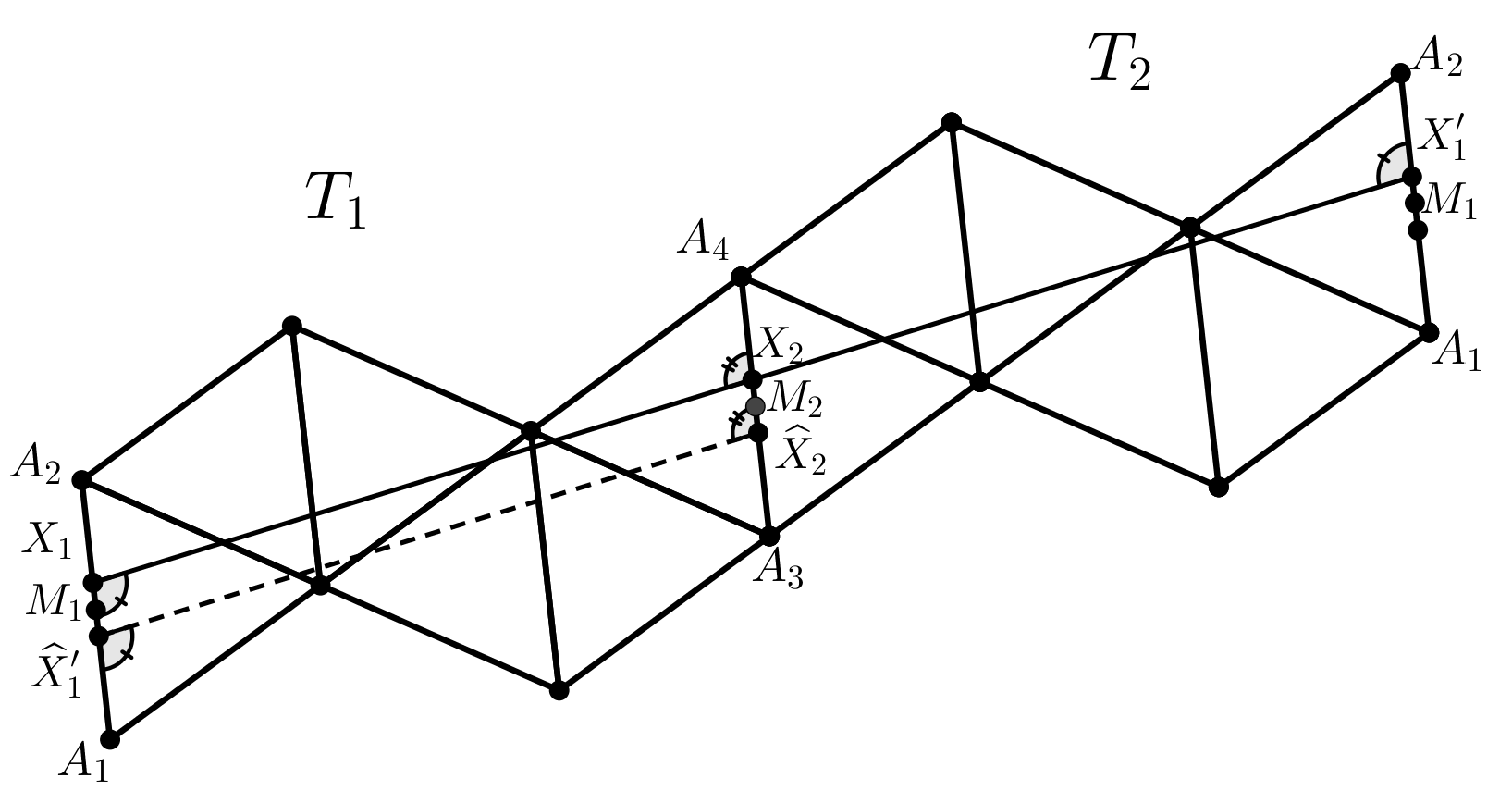}
\caption{ }
\label{uniqueness}
\end{center}
\end{figure}

If the $X_1X_2$ and  $\widehat{ X}'_1\widehat{ X}_2$ at the development don't intersect,
 then they form the  quadrilateral $X_1X_2\widehat {X}_2\widehat { X}'_1$  on the unite sphere.
 Since $\gamma$ is closed geodesic, then $\angle A_1X_1X_2+\angle A_2\widehat { X}'_1\widehat { X}'_2 = \pi$.
 Furthermore, $\angle X_1X_2A_3+ \angle\widehat {  X}'_1\widehat { X}_2A_4 =\pi$.
 We obtain the convex  quadrilateral on a sphere with  the sum of inner angles  $2\pi$ 
 (see Figure \ref{uniqueness}).
It follows that the integral of the Gaussian curvature over the interior of $X_1X_2\widehat {X}_2\widehat { X}'_1$
on a sphere is equal zero.
Hence, the segments $X_1X_2$ and   $\widehat { X}'_1\widehat { X}_2$ coincide
 under the symmetry of the development. 
 We obtain that the points $X_1$ and $X_2$ of geodesic $\gamma$  are the midpoints 
  of the edges $A_1A_2$ and $A_3A_4$ respectively.
  
  Similarly it can be proved that  $\gamma$ intersect the midpoints 
  of the second pair of the opposite edges of the tetrahedron. 
  \end{proof}

\begin{cor}\label{uniqueness_cor}
If two simple closed geodesic on a regular tetrahedron in spherical space intersect
 the edges of the tetrahedron in the same order, then they coincide. 
\end{cor}

\section{Lower bound of the length of simple closed geodesic on a regular tetrahedron in  $\Sp^3$}
 
 \begin{lemma}\label{length_below}
 On a regular tetrahedron with the faces angle  $\alpha$ in spherical space 
 the length of a simple closed geodesic of type  $(p,q)$ satisfies the inequality
\begin{equation}\label{estim_below}
L_{p,q} > 2 \sqrt{p^2+pq+q^2} \frac{\sqrt{4\sin^2\frac{\alpha}{2} - 1}}{\sin\frac{\alpha}{2}}.
\end{equation}
\end{lemma}

\begin{proof}

Consider a regular tetrahedron $A_1 A_2 A_3 A_4$ in $\Sp^3$ with the faces angle  
 $\alpha$ and $\gamma$  is a simple closed geodesic  of type  $(p,q)$ on it.
 
 Each of the faces of this tetrahedron is a regular spherical triangle.
 Consider a two-dimensional unit sphere containing the face $A_1A_2A_3$.
Construct the Euclidean plane $\Pi$ passing through the points  $A_1$, $A_2$ and $A_3$.
 The intersection of the sphere with $\Pi$  is a small circle. 
A ray starting at the sphere’s center $O$ and going through a point at the triangle $A_1A_2A_3$ 
  intersects the plane $\Pi$. 
So we get the geodesic map between the sphere and the plane  $\Pi$.
 The image of the spherical triangle $A_1A_2A_3$ is the triangle   $\bigtriangleup A_1A_2A_3$ 
 at the Euclidean plane  $\Pi$.
The edges of $\bigtriangleup A_1A_2A_3$  are the chords joining the vertices of the spherical triangle. 
From (\ref{a}) it follows that the length  $\widetilde{a}$ of the plane triangle’s edge equals
\begin{equation}\label{tilde_a}
\widetilde {a }= \frac{  \sqrt{ 4\sin^2\frac{\alpha}{2} - 1 }  }{\sin \frac{\alpha}{2}  }.
\end{equation}
The segments of the geodesic $\gamma$ lying inside  $A_1A_2A_3$ are mapped into
 the straight line segments inside  $\bigtriangleup  A_1A_2A_3$ (see Figure \ref{projection_inside}).

 \begin{figure}[h]
\begin{center}
\includegraphics[width=65mm]{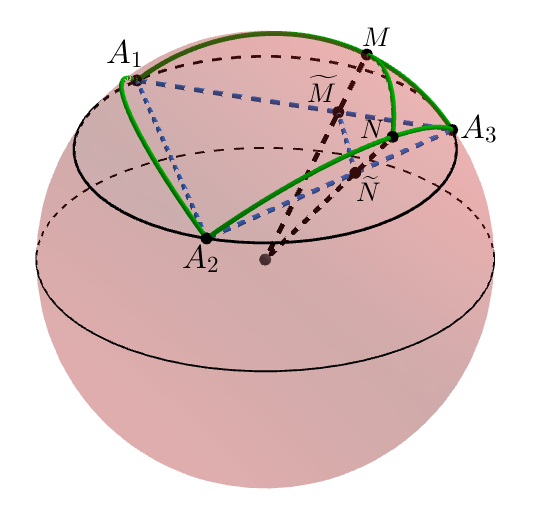}
\caption{ }
\label{projection_inside}
\end{center}
\end{figure}

In the similar way the other tetrahedron faces $A_2A_3A_4$, $A_2A_4A_1$ and $ A_1A_4A_3$ are mapped into 
the plane triangles $\bigtriangleup A_2A_3A_4$, 
$\bigtriangleup A_2A_4A_1$ and  $\bigtriangleup A_1A_4A_3$ respectively. 
 Since the spherical tetrahedron is regular, the constructed plane triangles are equal. 
 We can glue them together identifying the edges with the same labels. 
 Hence we obtain the regular tetrahedron in the Euclidean space. 
 Since the segments of $\gamma$ are mapped into the straight line segments within the plane triangles, 
 then they form the generalized  geodesic $\widetilde{ \gamma}$  
 on the regular Euclidean tetrahedron. 
Furthermore   $\widetilde{ \gamma}$ is equivalent to $\gamma$, so
 $\widetilde{ \gamma}$   passes through the tetrahedron’s edges
 at the same order as the simple closed geodesic of type $(p,q)$.  

Let us show that the length of $\gamma$ is greater than the length of $\widetilde{ \gamma}$.
 Consider an arc $MN$ of the geodesic $\gamma$ within the face  $A_1A_2A_3$.
  The rays $OM$ and $ON$  intersect the plane $\Pi$ at the points $\widetilde{ M}$ and $\widetilde{ N}$  respectively. 
  The   line segment $\widetilde{ M}$ и $\widetilde{ N}$ lying into $ \bigtriangleup  A_1A_2A_3$
 is the image of the arc $MN$ under the geodesic map   (see Figure~\ref{projection_inside}).
 Suppose that the length of the arc $MN$ is equal to $2\phi$, 
 then the length of the segment  $\widetilde{M}\widetilde{N}$  equals $2 \sin \phi$. 
Thus the length of  $\gamma$ on a regular tetrahedron in spherical space is greater than
 the length of its image   $\widetilde{\gamma}$ on a regular tetrahedron in Euclidean space.

From  Proposition~\ref{allgeod} we know that on a regular tetrahedron in  Euclidean space
  there exists a simple closed geodesic $\widehat {\gamma}$ 
   such that $\widehat {\gamma}$ is equivalent to  $\widetilde{ \gamma}$. 
Since on the development of the tetrahedron the geodesic  $\widehat {\gamma}$ is a line segment, and
the generalized  geodesic   $\widetilde{ \gamma}$ is a broken line, 
then the length of  $\widehat {\gamma}$ is less than the length of $\widetilde{\gamma}$.

Hence we get, that on a regular tetrahedron $A_1 A_2 A_3 A_4$ in $\Sp^3$ with the faces angle 
 $\alpha$ the length $L_{p,q}$ of a simple closed geodesic   $\gamma$ of type $(p,q)$  is greater than the length of a
 simple closed geodesic $\widehat {\gamma}$ of type $(p,q)$ on a regular tetrahedron in Euclidean space 
 with the edge length  $\widetilde{ a}$.
 From the equations  (\ref{length_evcl_geod}) and (\ref{tilde_a})  we get, that
\begin{equation} 
L_{p,q} > 2 \sqrt{p^2+pq+q^2} \frac{\sqrt{4\sin^2\frac{\alpha}{2} - 1}}{\sin\frac{\alpha}{2}}. \notag
\end{equation}
  \end{proof}
 
  \begin{theorem}\label{part1}
  On a regular tetrahedron with the faces angle   $\alpha$ in spherical space such that
 \begin{equation}\label{estim_above} 
\alpha > 2\arcsin \sqrt{ \frac{p^2+pq+q^2}{4(p^2+pq+q^2)-\pi^2} }, 
\end{equation}
where $(p,q)$ is a pair of coprime integers, 
there is no simple closed geodesic of type $(p,q)$.
\end{theorem}

 \begin{proof}
From Lemma~\ref{length} and the inequality  (\ref{estim_below}) we get, that if $\alpha$ fulfills the inequality
\begin{equation}\label{alpha_first} 
2 \sqrt{p^2+pq+q^2} \frac{\sqrt{4\sin^2\frac{\alpha}{2} - 1}}{\sin\frac{\alpha}{2}} > 2\pi,  
\end{equation} 
then the necessary condition for the existence of a simple closed geodesic of type $(p,q)$
 on a regular tetrahedron with face’s angle  $\alpha$  in spherical space is failed. 
 Hence, after modifying  (\ref{alpha_first}), we get that, if 
 \begin{equation}\label{ } 
\alpha > 2\arcsin \sqrt{ \frac{p^2+pq+q^2}{4(p^2+pq+q^2)-\pi^2} },  \notag
\end{equation}
  then there is no simple closed geodesics of type $(p,q)$ 
  on the tetrahedron with face’s angle  $\alpha$  in spherical space. 
   \end{proof}

 \begin{cor}\label{cor_th_1}  
 On a regular tetrahedron in spherical space there exist a finite number of simple closed geodesics.
\end{cor}

 \begin{proof}
If the integers  $(p,q)$ go to infinity, then 
\begin{equation}
\lim\limits_{p,q \to\infty } 2\arcsin \sqrt{ \frac{p^2+pq+q^2}{4(p^2+pq+q^2)-\pi^2} } 
= 2\arcsin \frac{1}{2} = \frac{\pi}{3}. \notag
\end{equation}
From the inequality  (\ref{estim_above}) we get, that 
for the large numbers  $(p,q)$ a simple closed geodesic of type $(p,q)$
could exist on a regular tetrahedron with the faces angle $\alpha$ closed to $ \pi/3$ in spherical space.
\end{proof}
 
 The pairs $p=0, q=1$ and $p=1, q=1$ don't satisfy the condition (\ref{estim_above}). 
 Geodesics of this types are considered in Section  \ref{exampl}.

\section{Sufficient condition for the  existence of a simple closed geodesic on a regular tetrahedron in $\Sp^3$}
 
In previous sections we assumed that the Gaussian curvature of the tetrahedron's faces is equal $1$ in 
spherical space. 
In this case  faces of a tetrahedron were regular spherical triangles with  angles $\alpha$
on a unit two-dimensional  sphere.
The length $a$ of the edges was the function depending on $\alpha$ according to  (\ref{a}).
In current section we will assume that the tetrahedron's faces are spherical triangles with the  angle $\alpha$
on a sphere of radius $R = 1/a$.
In this case the length of the tetrahedron edges equals $1$, and the faces curvature is $a^2$.

Since $\alpha >\pi/3$, then we can assume $\alpha =\pi/3 + \varepsilon$, where $\varepsilon > 0$. 
Taking into account Lemma~\ref{morePi2} we also expect $\varepsilon < \pi/6$.
Let us proof some subsidiary results.

  \begin{lemma}\label{estim_a_above}
  The length of the edges of a regular tetrahedron in spherical space of curvature one satisfies inequality 
   \begin{equation}\label{a_above_epsilon}
 a <  \pi  \sqrt{2  \cos \frac{\pi}{12} } \cdot \sqrt{  \varepsilon  },        
 \end{equation}
where $\alpha=\frac{\pi}{3}+\varepsilon$ is a value of  faces angles.  
   \end{lemma}
 
\begin{proof} 
From the equality (\ref{a}) it follows, that
\begin{equation}
\sin a = \frac{ \sqrt{4 \sin^2 \frac{\alpha}{2} -1} }{ 2 \sin^2 \frac{\alpha}{2} }. \notag 
\end{equation}
Substituting $\alpha = \frac{\pi}{3} + \varepsilon$ into this equality, we get
\begin{equation}
\sin a = 
\frac{ \sqrt{\sin \frac{\varepsilon}{2} \cos \left( \frac{\pi}{6} - \frac{\varepsilon}{2} \right) } }
{ \sin^2 \left( \frac{\pi}{6} + \frac{\varepsilon}{2} \right) }. \notag 
\end{equation}
Since $\varepsilon < \frac{\pi}{6} $, then
$\cos \left( \frac{\pi}{6} - \frac{\varepsilon}{2} \right) < \cos \frac{\pi}{12} $,
$\sin \left( \frac{\pi}{6} + \frac{\varepsilon}{2} \right) > \sin \frac{\pi}{6} $
and $\sin \frac{\varepsilon}{2} < \frac{\varepsilon}{2}$.
Using this estimations, we obtain
\begin{equation}\label{sin_a}
\sin a < 2 \sqrt{2 \cos \frac{\pi}{12} } \cdot \sqrt{ \varepsilon } . 
\end{equation}
Since $a < \frac{\pi}{2}$, then $\sin a > \frac{2}{\pi} a $. 
Substituting this estimation into (\ref{sin_a}), we get the necessary inequality (\ref{a_above_epsilon}).
\end{proof}

Consider a parametrization of a two-dimensional sphere $S^2$ of radius $R $ in three-dimensional Euclidean case:
  \begin{equation}\label{spher_th2}
 \begin{cases}
x=R \sin \phi  \cos \theta \\
y=R  \sin \phi  \sin \theta \\
z=- R\cos \phi 
 \end{cases},  
 \end{equation}
where $\phi \in [0, \pi]$, $\theta \in [0, 2\pi)$. 
Let the point $P$ have the coordinates $\phi = r/R, \theta = 0$, where $ r/R< \pi/2$,
and the point $X_1$ correspond to  $\phi = 0$.
Apply a central projection of the hemisphere $\phi \in [0,  \pi/2]$, $\theta \in [0, 2\pi)$  
 to the tangent plane at the point $X_1$. 
 
  \begin{lemma}\label{estim_alpha_above}
 The central projection  of the hemisphere of radius $R = 1/a$ to the tangent  plane at the point $X_1$
 maps the angle   $\alpha = \pi/3 + \varepsilon$ with the vertex  $P(R\sin \frac{r}{R}, 0, -R\cos \frac{r}{R})$ 
 at this hemisphere to the angle    $\widehat{\alpha}_r$ on a plane, which satisfies the inequality 
        \begin{equation}\label{bar_alpha-pi_3}
  \Big| \widehat{\alpha}_r-  \frac{\pi}{3} \Big| < \pi \tan^2 \frac{r}{R} + \varepsilon.   
  \end{equation}
 \end{lemma}

\begin{proof}
Construct two planes $\Pi_1$ and $\Pi_2$ passing through the center of a hemisphere and 
the point $P(R\sin \frac{r}{R}, 0, -R\cos \frac{r}{R})$:
  \begin{equation}
\Pi_1 : a_1 \cos \frac{r}{R} \: x + \sqrt{1 -a_1^2}\:\: y + a_1 \sin\frac{r}{R} \:  z = 0; \notag
  \end{equation}
    \begin{equation}
 \Pi_2 : a_2 \cos \frac{r}{R}\:  x + \sqrt{1-a_2^2} \:\: y + a_2 \sin\frac{r}{R} \: z = 0, \notag
  \end{equation}
where 
      \begin{equation}\label{a_1a_2}
     |a_1|, |a_2|  \le 1.
        \end{equation}
Let the angle between this two planes   $\Pi_1$ and $\Pi_2$ equals $\alpha$. Then
     \begin{equation} \label{cos_alpha_sphere}
\cos \alpha =   a_1a_2 + \sqrt{  (1 -a_1^2)(1 -a_2^2) }.
  \end{equation}
  
 \begin{figure}[h]
\begin{center}
\includegraphics[width=75mm]{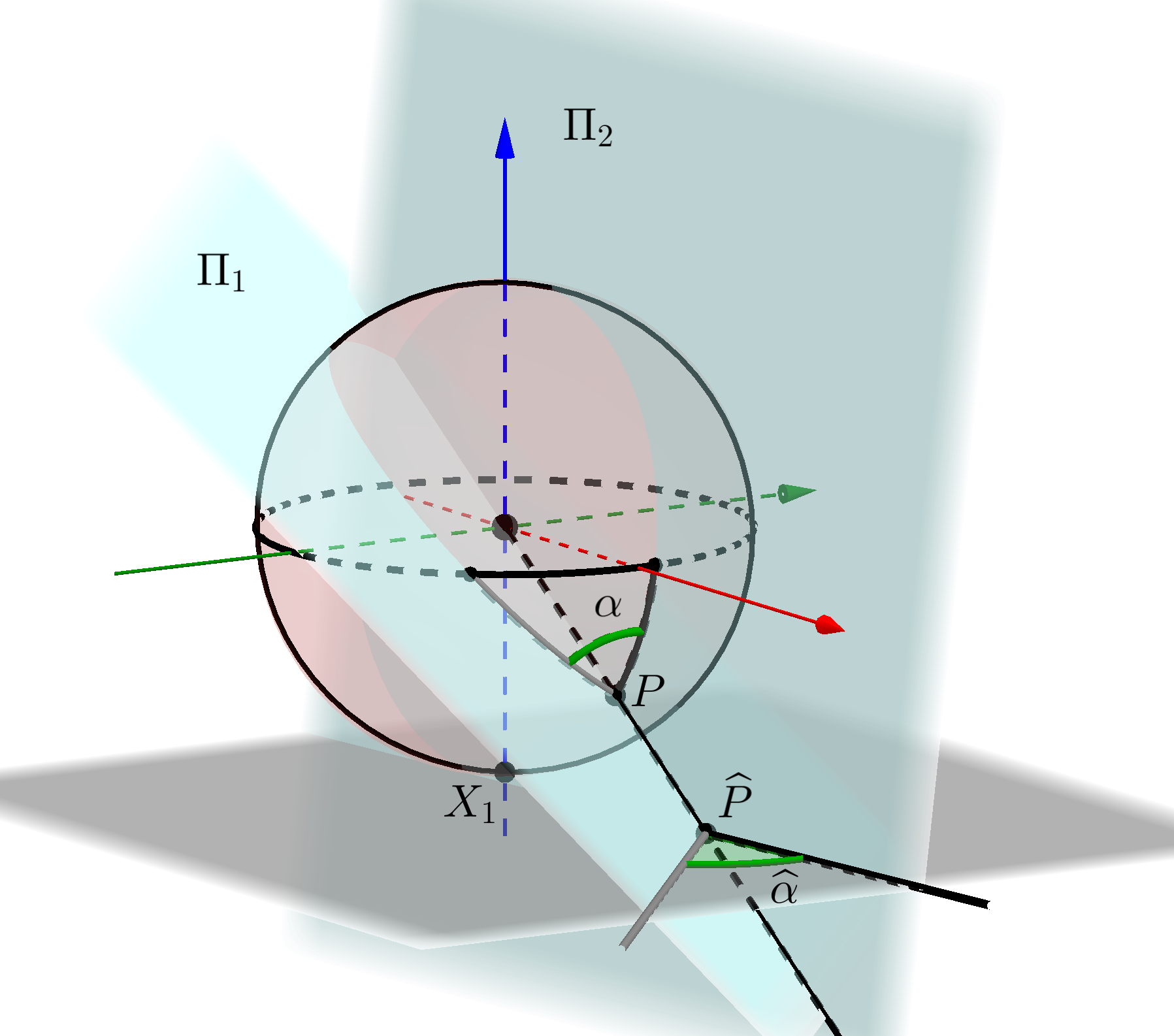}
\caption{ }
\label{angle_projection}
\end{center}
\end{figure}

The equation of the tangent plane at the  $X_1$ to the sphere $S^2$ is equal  $z=-R$.
This tangent plane intersects  the planes  $\Pi_1$ and $\Pi_2$ along the lines, 
which form the angle  $\widehat{\alpha}_r$  (see Figure \ref{angle_projection}),   and
     \begin{equation} \label{cos_alpha_bar}
\cos \widehat{\alpha}_r = \frac{  a_1a_2 \cos^2\frac{r}{R} + \sqrt{ (1 -a_1^2)(1 -a_2^2) }   }
{ \sqrt{1 -a_1^2 \sin^2\frac{r}{R}  }   \sqrt{1 -a_2^2 \sin^2\frac{r}{R}  }    }.
  \end{equation}
From the equations (\ref{cos_alpha_sphere}) and (\ref{cos_alpha_bar}) we get 
     \begin{equation}\label{vspomagat_1}
  | \cos \widehat{\alpha}_r- \cos \alpha | < 
  \frac{ |   a_1a_2 \sin^2\frac{r}{R} | }{ \sqrt{1 -a_1^2 \sin^2\frac{r}{R}  } \sqrt{1 -a_2^2 \sin^2\frac{r}{R}  }  }. 
  \end{equation}
  Applying the inequality (\ref{a_1a_2}) and the estimation (\ref{vspomagat_1}) we get 
\begin{equation}\label{vspomagat_2}
| \cos  \widehat{\alpha}_r - \cos \alpha | < \tan^2 \frac{r}{R}. 
\end{equation}
From the equation
\begin{equation}
| \cos \widehat{\alpha}_r- \cos \alpha | = 
\Big|2 \sin  \frac{ \widehat{\alpha}_r - \alpha}{2} \sin \frac{ \widehat{\alpha}_r + \alpha}{2} \Big|  \notag
 \end{equation}
and inequalities
\begin{equation}
 \left |  \sin  \frac{ \widehat{\alpha}_r+ \alpha}{2}  \right| > \sin \frac{\pi}{6}, 
 \;\;\;
 \left |  \sin  \frac{ \widehat{\alpha}_r- \alpha}{2} \right|
  > \frac{2}{\pi}  \left | \frac{ \widehat{\alpha}_r- \alpha}{2} \right|,  \notag
  \end{equation}
  and since $\alpha > \frac{\pi}{3}$ and $ \widehat{\alpha}_r < \pi $ it follows 
      \begin{equation}\label{bar_alpha_step1} 
     \frac{2}{\pi}  \left | \frac{ \widehat{\alpha}_r - \alpha}{2} \right| < | \cos \widehat{\alpha}_r - \cos \alpha |.  \notag
  \end{equation}
From (\ref{bar_alpha_step1}), (\ref{vspomagat_2}) and equality $\alpha = \frac{\pi}{3} + \varepsilon$ we obtain
\begin{equation}\label{bar_alpha_pi_3_1}
\Big| \widehat{\alpha}_r- \frac{\pi}{3} \Big| < \pi \tan^2 \frac{r}{R} + \varepsilon.  \notag
\end{equation}
 \end{proof}

Let us consider the arc of length one starting at the point  $P$ with the coordinates 
 $\phi =r/R, \theta = 0$, where $r/R< \pi/2 $, on a sphere  (\ref{spher_th2}).
 Apply the central projection of this arc to the  plane $z=-R$, which is tangent to the sphere at the point $X_1 (\phi = 0)$.

  \begin{lemma}\label{estim_l_above}
   The central projection  of the hemisphere of radius $R = 1/a$ to the tangent  plane at the point $X_1$
 maps the arc of the length one starting from the point  $P(R\sin \frac{r}{R}, 0, -R\cos \frac{r}{R})$ 
 into the segment of length  $\widehat{l}_r$ satisfying the inequality 
        \begin{equation}\label{bar_l-1}
\widehat{l}_r-1 < \frac {  \cos \frac{\pi}{12}  \left( 4+  \pi^2  (2r+1)^2  \right) }
{  \left( 1-\frac{2}{\pi}a(r+1) \right)^2 } \cdot \varepsilon.     
  \end{equation}
 \end{lemma}

\begin{proof}
The point $P(R\sin \frac{r}{R}, 0, -R\cos \frac{r}{R})$  on the sphere  $S^2$ is projected
 to the point  $\widehat{P}( R \tan \frac{r}{R}, 0, -R)$ on the tangent plane $z=-R$.

Take the point $Q(R a_1, R a_2, R a_3)$ on a sphere such that the spherical distance $PQ$ equals~$1$.
Then $\angle POQ = 1/R $, where $O $ is a center of the sphere $S^2$ (see Figure~\ref{length_projection}).
Thus we get the following conditions for the constants $a_1, a_2, a_3$:
  \begin{equation}\label{cond_1}
  a_1 \sin\frac{r}{R} - a_3 \cos \frac{r}{R} = \cos \frac{1}{R};
  \end{equation}
   \begin{equation}\label{cond_2}
  a_1^2+a_2^2+a_3^2 = 1.
  \end{equation}

The central projection into the plane $z=-R$ maps the point $Q $ to the point\\
$\widehat{Q} \left(- \frac{a_1}{a_3} R, - \frac{a_2}{a_3} R, -R \right)$.
The length of $\widehat{P}\widehat{Q}$ equals
\begin{equation}\label{P'Q'}
|\widehat{P}\widehat{Q}| =R \sqrt{ \left( \frac{a_1}{a_3} - \tan \frac{r}{R} \right)^2 + \frac{a^2_2}{a_3^2} }
\end{equation}

Applying the method of Lagrange multipliers for finding the local extremum of the length  $\widehat{P}\widehat{Q}$,
we get, that the minimum of $|\widehat{P}\widehat{Q}| $ reaches when 
 $Q$ has the coordinates $\left( R\sin\frac{r-1}{R}, 0, R\cos\frac{r-1}{R} \right)$. Then
   \begin{equation}\label{ }
 |\widehat{P}\widehat{Q}|_{min} = R     \left|  \tan \frac{r}{R}  - \tan \frac{r-1}{R}  \right|  =  
   \frac {R \sin\frac{1}{R} }{\cos\frac{r}{R} \cos \frac{r-1}{R}}.   \notag
  \end{equation}
 Note  that $ |\widehat{P}\widehat{Q}|_{min}  > 1$. 
 
 \begin{figure}[h]
\begin{center}
\includegraphics[width=80mm]{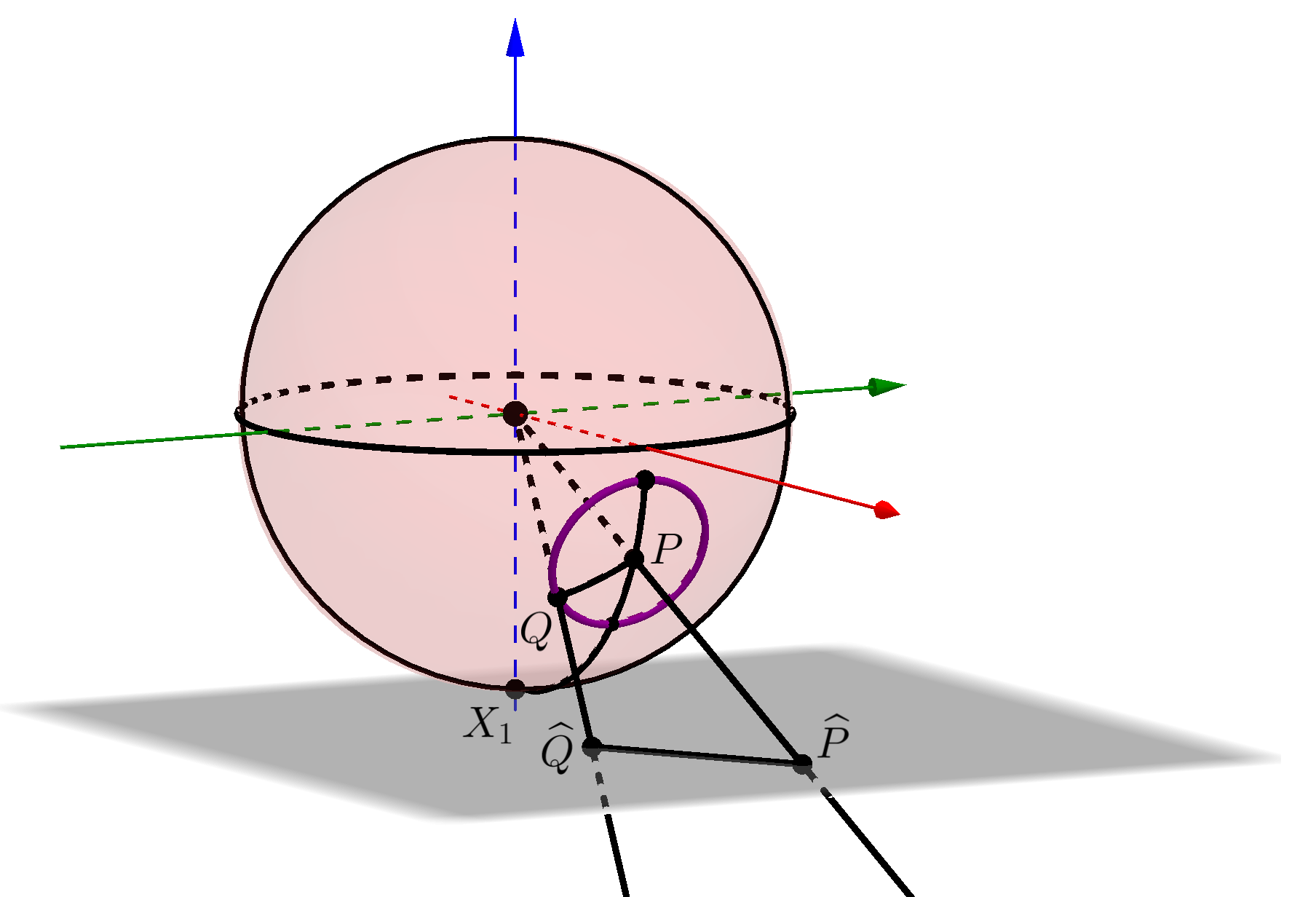}
\caption{ }
\label{length_projection}
\end{center}
\end{figure}

  The maximum of $|\widehat{P}\widehat{Q}| $ reaches at the point $Q\left(R\sin\frac{r+1}{R}, 0, R\cos\frac{r+1}{R} \right)$. 
  This maximum value equals 
   \begin{equation}\label{ }
 |\widehat{P}\widehat{Q}|_{max} = R     \left|  \tan \frac{r}{R}  - \tan \frac{r+1}{R}  \right|  =  
   \frac {R \sin\frac{1}{R} }{ \cos\frac{r}{R} \cos \frac{r+1}{R}  }.   \notag
  \end{equation}
Since $R =1/a$, then the length $\widehat{l}_r$ of the projection of $PQ$ satisfies 
\begin{equation}\label{bar_l}
\widehat{l}_r< \frac { \sin a }{ a   \cos (a r) \cos \big( a (r+1)\big) }. \notag
\end{equation}
Applying the estimation   $\sin a<a$,  we obtain 
\begin{equation}\label{vspom_bar_l}
\widehat{l}_r-1 < \frac { 2 -  \cos a  - \cos   \big( a(2r+1) \big) }{2   \cos (a r) \cos \big( a (r+1)\big)  }. 
\end{equation}
From the equation (\ref{sin_a}) it follows 
\begin{equation}\label{vsp_1}
1 - \cos a = \frac{\sin^2a}{1+ \cos a} \le 8 \cos \frac{\pi}{12}\; \varepsilon.
\end{equation}
In the same way from the inequality  (\ref{a_above_epsilon}) we obtain 
\begin{equation}\label{vsp_2}
1- \cos \left( a (2r+1) \right) \le 2\pi^2 \cos \frac{\pi}{12} (2r+1)^2 \; \varepsilon;
\end{equation}   
Estimate the denominator of the   (\ref{vspom_bar_l}), using the inequality
 $\cos x > 1-\frac{2}{\pi}x$ where $x <\pi/2$.
Applying the inequalities  (\ref{vsp_1}) и (\ref{vsp_2}), we get
\begin{equation}\label{ }
\widehat{l}_r -1 < \frac {   4 \cos \frac{\pi}{12} + \pi^2 \cos \frac{\pi}{12} (2r+1)^2}
{    \left( 1- \frac{2}{\pi} a \left( r+1 \right) \right)^2  } \cdot \varepsilon.  \notag 
 \end{equation}
   \end{proof}

 \begin{theorem}
Let $(p,q)$ be a pair of coprime integers, $0\le p <q$,  and let  $ \varepsilon$ satisfy
    \begin{equation}\label{bar_h_1} 
     \varepsilon  <  \min \left\{
      \frac{\sqrt{3}}{4  c_0\sqrt{p^2+q^2+pq}\;
 \sum_{i=0}^{\left[ \frac{p+q}{2} \right]+2}   \left( c_l (i) + \sum_{j=0}^i  c_\alpha (j) \right)};
    \frac{1}{8  \cos \frac{\pi}{12}(p+q)^2 }
  \right\},    
 \end{equation}
where
   \begin{equation}\label{  }
 c_0=   \frac{ 
 3 - \frac{(p+q+2)}{\pi  \cos \frac{\pi}{12}  (p+q)^2 } - 
   16 \sum_{i=0}^{\left[ \frac{p+q}{2} \right]+2}  \tan^2 \left( \frac{\pi i}{2 (p+q) } \right) }
  {1-  \frac{(p+q+2)}{2\pi  \cos \frac{\pi}{12}  (p+q)^2 } -
  8 \sum_{i=0}^{\left[ \frac{p+q}{2} \right]+2}  \tan^2 \left( \frac{\pi i}{2 (p+q) } \right)  }, \notag
 \end{equation} 
     \begin{equation}\label{ }
 c_l (i) =\frac { \cos \frac{\pi}{12} (p+q)^2  \left( 4+  \pi^2  (2i+1)^2  \right) }{ \left( p+q-i-1 \right)^2 }, \notag
 \end{equation}
    \begin{equation}\label{ }
 c_\alpha (j)  =4   \left(   8 \pi     (p+q)^2 \cos \frac{\pi}{12}  \tan^2 \frac{\pi j}{2 (p+q) }  + 1 \right). \notag
 \end{equation}
 Then on a regular tetrahedron in spherical space with the faces angle $\alpha= \pi/3+\varepsilon$
there exists a unique, up to the rigid motion of the tetrahedron, simple closed geodesic of type  $(p,q)$.
 \end{theorem}

  \begin{proof}
Take the pair of coprime integers $(p,q)$, where  $0< p <q$.
Consider a simple closed geodesic $\widetilde{\gamma}$ of type  $(p,q)$ 
on a regular tetrahedron $\widetilde {A}_1\widetilde{A}_2\widetilde{ A}_3\widetilde{ A}_4$  with 
the  edge of the length one in Euclidean space.
Suppose, that $\widetilde{\gamma}$ passes through the midpoints  
$\widetilde{ X}_1$, $\widetilde{ X}_2$ and $\widetilde{ Y}_1$, $\widetilde{ Y}_2$
of the edges $\widetilde{ A}_1\widetilde{A}_2$ и $\widetilde{ A}_3\widetilde{A}_4$
and $\widetilde{A}_1\widetilde{A}_3$, $\widetilde{A}_4\widetilde{ A}_2$ respectively.

Consider the development  $\widetilde{T}_{pq}$ of the tetrahedron 
along  $\widetilde{ \gamma}$ starting from the point $\widetilde{ X}_1$.
The geodesic unfolds to the segment 
$\widetilde{ X}_1\widetilde{ Y }_1\widetilde{ X}_2\widetilde{ Y}_2  \widetilde{X'_1}$ 
inside the development  $\widetilde{T}_{pq}$.
From Proposition~\ref{corparts2} we know, that the parts of the development along 
geodesic segments $\widetilde{X}_1\widetilde{Y}_1$, $\widetilde{Y}_1\widetilde{X}_2$,
 $\widetilde{X}_2\widetilde{Y}_2$ and $ \widetilde{Y}_2\widetilde{X}'_1$ are equal,
 and any two adjacent polygons can be transformed into each other by a rotation through an angle $\pi$
  around the midpoint of their common edge  (see Figure \ref{geod_23_evcl}).

\begin{figure}[h]
\columnsep=50pt
\begin{multicols}{2}
\hfill
\centering{\includegraphics[width=100mm]{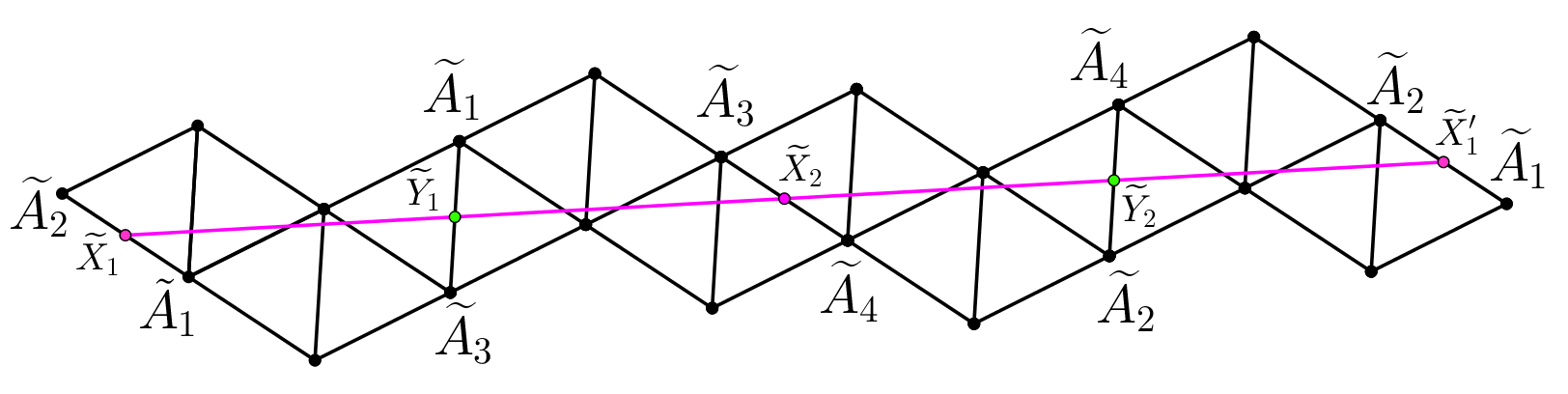}  }
\caption{ }
\label{geod_23_evcl}
\centering{\includegraphics[width=60mm]{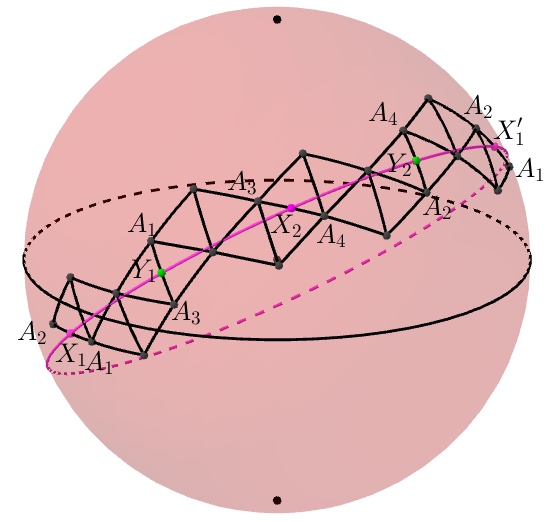} }
\caption{ }
\label{geod_23_spher}
\end{multicols}
\end{figure}

Now consider a two-dimensional sphere $S^2$ of radius $R = 1/a$,
 where  $a$ depends on $\alpha$ according to (\ref{a}).
On this sphere we take the several copies of regular spherical triangles with the angle $\alpha$ at vertices,
 where  $\pi/3 <\alpha <\pi/2$.
Put this triangles up in the same order as we develop the faces of the Euclidean tetrahedron 
along $\widetilde{\gamma}$ into the plane.
In other words, we construct a polygon $T_{pq}$ on a sphere $S^2$
formed by the same sequence of regular triangles as the polygon $\widetilde{T}_{pq}$ on   Euclidean plane.
 Denote the vertices of $T_{pq}$ according to the vertices of $\widetilde{T}_{pq}$.
By the construction  the spherical polygon $T_{pq}$ has the same properties of the central symmetry 
as the Euclidean  $\widetilde{T}_{pq}$.
Since the groups of isometries of   regular tetrahedra in spherical space and in Euclidean space are equal, 
then $T_{pq}$ corresponds to the development of a regular tetrahedron with the faces angle   $\alpha$ in
spherical space.

Denote by $X_1$, $X'_1$ and $X_2$, $Y_1$, $Y_2$ the midpoints of the edges 
$A_1A_2$, $A_3A_4$,  $A_1A_3$,  $A_4A_2$ on $T_{pq}$  such that these midpoints correspond to the points 
 $\widetilde{X}_1$, $\widetilde{X}'_1$ and $\widetilde{X}_2$, $\widetilde{Y}_1$, $\widetilde{Y}_2$ 
 on the Euclidean development  $\widetilde{T}_{pq}$.
Construct the great arcs $ X_1Y_1$, $Y_1X_2$, $ X_2Y_2$ и $ Y_2X'_1$ on a sphere.
From  the  properties of the central symmetry of $T_{pq}$ we obtain that these arcs form the one great arc   $X_1X'_1$
on  $S^2$   (see Figure \ref{geod_23_spher}).
Since the polygon $T_{pq}$ is not convex, 
we want to find  $\alpha$  such that the polygon $T_{pq}$ contains the arc  $X_1Y_1$  inside.
Therefore the whole arc  $X_1X'_1$ will be also contained  inside  $T_{pq}$ and 
 $X_1X'_1$ will correspond  to the  simple closed geodesic of type  $(p,q)$ on a regular tetrahedron with the 
 faces angle  $\alpha$ in spherical space.

In the following we will consider the part of the polygon  $T_{pq}$  along $X_1Y_1$ and 
will also denote it as $T_{pq}$.
This part consists of $p+q$ regular spherical triangles with the edges of length one. 
Then if $a$ satisfies the following inequality
\begin{equation}\label{a_p+q}
a (p+q) < \frac{\pi}{2},
\end{equation}
then the  polygon $T_{pq}$ is contained inside the open hemisphere.
Since $\alpha = \pi/3+\varepsilon$, then from the condition  $(\ref{a_above_epsilon})$ we get that
the estimation (\ref{a_p+q}) holds if
\begin{equation}\label{epsilon_estim_1}
\varepsilon < \frac{1}{8 \cos \frac{\pi}{12}(p+q)^2 }. 
\end{equation}
In this case the length of the arc $X_1Y_1$ is less than  $\pi/2 a$, 
so   $X_1Y_1$ satisfies the necessary condition from Lemma~\ref{length}.

Apply a central projection of the $T_{pq}$ into the tangent plane  $T_{X_1}S^2$ at the point $X_1$ to the sphere $S^2$. 
Since the central projection is a geodesic map, then the image of the spherical polygon  $T_{pq}$ on  $T_{X_1}S^2$
is a polygon  $\widehat{T}_{pq}$. 

Denote by  $\widehat{A}_i$ the vertex of $\widehat{T}_{pq}$, which is an image of the vertex $A_i$ on $T_{pq}$. 
The arc $X_1Y_1$ maps into the line segment $\widehat{X}_1\widehat{Y}_1$  on   $T_{X_1}S^2$, that
joins  the midpoints of the edges  $\widehat{ A}_1\widehat{ A}_2$ и $\widehat{ A}_1\widehat{ A}_3$.
If for some $\alpha$ the segment $\widehat{ X}_1\widehat{ Y }_1$ lies inside the polygon  $\widehat{ T}_{pq}$, then 
the arc $X_1Y_1$ is also containing inside  $T_{pq}$ on the  sphere. 

The vector $\widehat{ X}_1\widehat{Y}_1$ equals
     \begin{equation}\label{X_1_overline_Y_1 } 
\widehat{ X}_1\widehat{Y}_1= \widehat{a_0}+\widehat{a_1}+\dots+\widehat{ a_s}+\widehat{a}_{s+1},
 \end{equation}
 where $\widehat{a_i}$ are the sequential vectors of the $\widehat{T}_{pq}$  boundary,  
 $\widehat{a_0}=\widehat{ X_1}\widehat{ A_2}$,  $\widehat{a}_{s+1} =\widehat{ A_1}\widehat{Y_1}$,
and $s=\left[ \frac{p+q}{2} \right]+1$
  (if we take the boundary of $\widehat{ T}_{pq}$ from the other side of
 $\widehat{ X}_1\widehat{ Y}_1$, then $s=\left[ \frac{p+q}{2} \right]$ )
 (see Figure \ref{dev_compare}).
 
On the other hand  at the Euclidean plane $T_{X_1}S^2$ 
there exists a development $\widetilde{T}_{pq}$ of a regular Euclidean tetrahedron 
  $\widetilde{A}_1\widetilde{  A}_2\widetilde{ A}_3\widetilde{  A}_4$ with the edge of length one 
  along a simple closed geodesic  $\widetilde{\gamma}$.
The development  $\widetilde{T}_{pq}$  is equivalent to $T_{pq}$, and then  it's equivalent to  $\widehat{T}_{pq}$.
  The segment  $ \widetilde{  X}_1\widetilde{ Y}_1$ lies inside  $\widetilde{  T}_{pq}$
  and corresponds to the segment of $\widetilde{\gamma}$.
  
Let the development  $\widetilde{T}_{pq}$ be placed such that 
the point  $ \widetilde{X}_1$  coincides with $\widehat{X}_1 $  of  $\widehat{ T}_{pq}$, 
and the vector  $\widehat{ X}_1\widehat{ A}_2$ has the same direction with $\widetilde{X}_1\widetilde{ A}_2$.
Similarly the vector $\widetilde{ X}_1 \widetilde{Y}_1$ equals
\begin{equation}\label{tilde_X_1_tilde_Y_1} 
\widetilde{ X}_1\widetilde{Y}_1 = \widetilde{a_0}+\widetilde{a_1}+\dots+\widetilde{a_s}+\widetilde{a}_{s+1},
\end{equation}
where  $\widetilde{a_i}$ are  the sequential vectors of the $\widetilde T_{pq}$  boundary,  
   $s=\left[ \frac{p+q}{2} \right]+1$ and 
$\widetilde{a_0}=\widetilde{ X}_1 \widetilde{A}_2$, 
$\widetilde{a}_{s+1}= \widetilde{ A}_1 \widetilde{Y}_1$  (see Figure~\ref{dev_compare}).  

 \begin{figure}[h]
\begin{center}
\includegraphics[width=80mm]{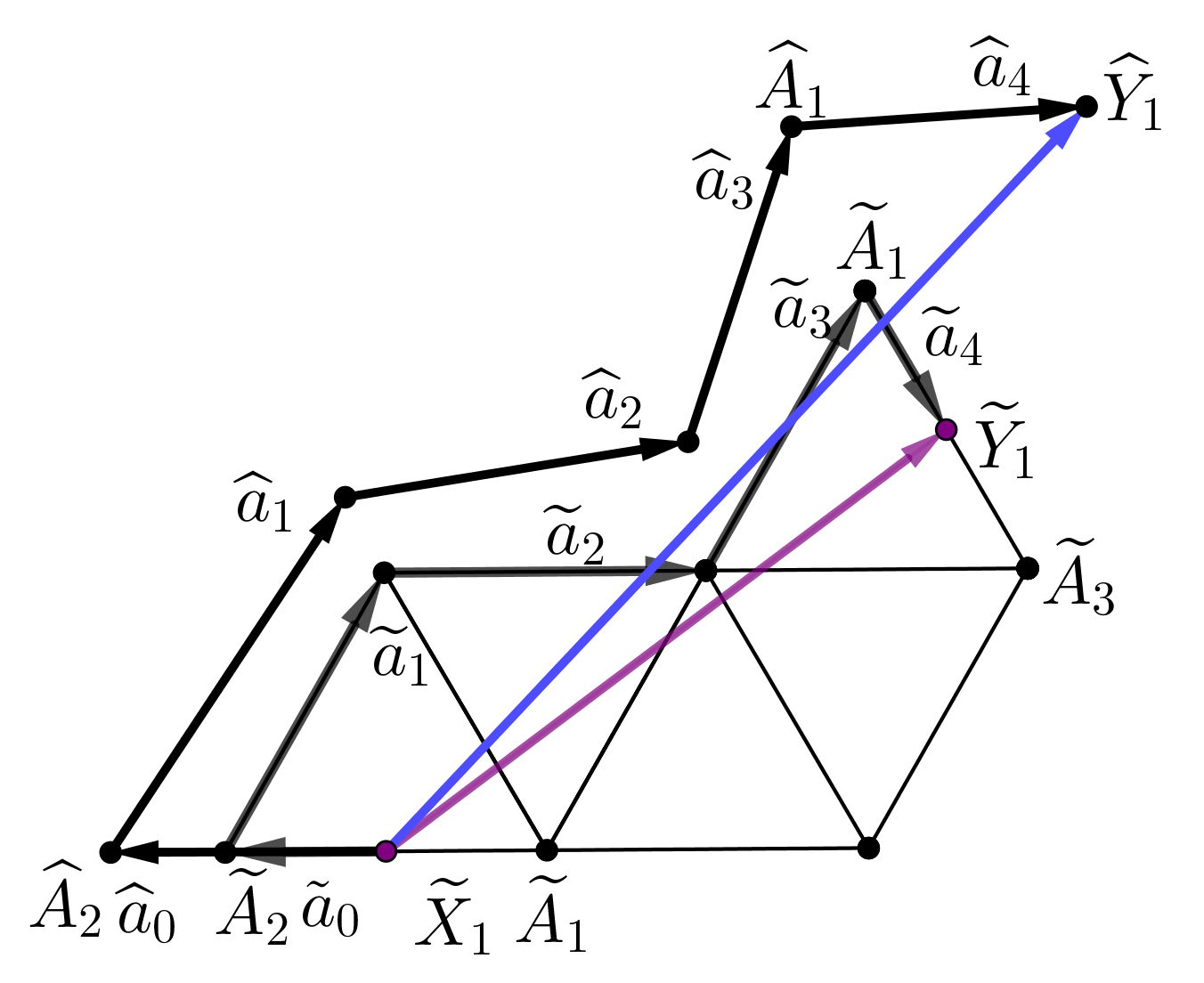}
\caption{ }
\label{dev_compare}
\end{center}
\end{figure}

Suppose the minimal distance from the vertices of $\widetilde{  T}_{pq}$  to the segment
$\widetilde{ X}_1 \widetilde{ Y}_1$ reaches at the vertex $ \widetilde{ A}_k$ and equals $\widetilde{ h}$
from the formula  (\ref{evcl_dist_vertex}).
Let us find the estimation of the distance $\widehat h$  between the segment   $\widehat { X}_1 \widehat {Y }_1$ 
and the corresponding vertex  $\widehat  {A }_k$ on  $\widehat{ T}_{pq}$.
 A geodesic on a regular tetrahedron in Euclidean space intersects at most three edges starting from the same
tetrahedron's vertex.
It follows, that the interior angles of the polygon $ \widetilde{ T}_{pq}$ is not   greater than $4 \frac{\pi}{3}$, 
and the angles of the corresponding vertices on $\widehat{T}_{pq}$  is not   greater than $4 \widehat{\alpha}_i$. 
Applying the formula (\ref{bar_alpha-pi_3}) for  $1\le i \le s$
we get, that the angle between $ \widehat{a_i}$ and $ \widetilde{a_i} $ satisfies the inequality 
 \begin{equation}\label{angle_b_ia_i_1}
\angle(\widehat{a_i}, \widetilde {a_i} ) < \sum_{j=0}^i 4 \left( \pi  \tan^2 \frac{j}{R} + \varepsilon \right).  
 \end{equation} 
Since $R=1/a$, then using  (\ref{a_above_epsilon})  we obtain  
 \begin{equation}\label{tg_r_R}
  \tan \frac{j}{R} <  \tan \left( j \pi  \sqrt{2  \cos \frac{\pi}{12} } \; \sqrt{  \varepsilon  }   \right).   
  \end{equation}
The inequality (\ref{a_p+q}) holds if  the following condition fulfills
   \begin{equation}\label{vspomagat_3}
  \tan \left( j \pi  \sqrt{2  \cos \frac{\pi}{12} } \; \sqrt{  \varepsilon  }   \right) <   \tan \frac{\pi j}{2 (p+q) }.   
  \end{equation}
If $\tan x < \tan x_0$, then $\tan x < \frac{\tan x_0}{x_0} x$. 
From  (\ref{vspomagat_3}) it follows 
 \begin{equation}\label{vspomagat_4}
\tan \left( j \pi  \sqrt{2  \cos \frac{\pi}{12} } \; \sqrt{  \varepsilon  }   \right) <  
 2 (p+q)  \tan   \frac{\pi j}{2 (p+q) }  \sqrt{2 \cos \frac{\pi}{12}} \;  \sqrt{\varepsilon}.  
\end{equation}
Therefore from (\ref{tg_r_R}) and  (\ref{vspomagat_4}) we get  
 \begin{equation}\label{vspomagat_5}
\tan \frac{j}{R} < 2 (p+q)  \tan   \frac{\pi j}{2 (p+q) }  \sqrt{2 \cos \frac{\pi}{12}} \;  \sqrt{\varepsilon}.  
\end{equation}
Using (\ref{angle_b_ia_i_1}) and (\ref{vspomagat_5}) we obtain the final estimation for the angle between 
the vectors $\widehat{a_i}$ and $ \widetilde{a_i}$:
\begin{equation}\label{angle_b_ia_i}
\angle(\widehat{a_i}, \widetilde {a_i} ) < \sum_{j=0}^i  4
  \left(   8 \pi     (p+q)^2 \cos \frac{\pi}{12}  \tan^2 \frac{\pi j}{2 (p+q) }  + 1 \right)  \varepsilon.       
  \end{equation}

\begin{figure}[h]
\columnsep=5pt
\begin{multicols}{2}
\hfill
\centering{\includegraphics[width=60mm]{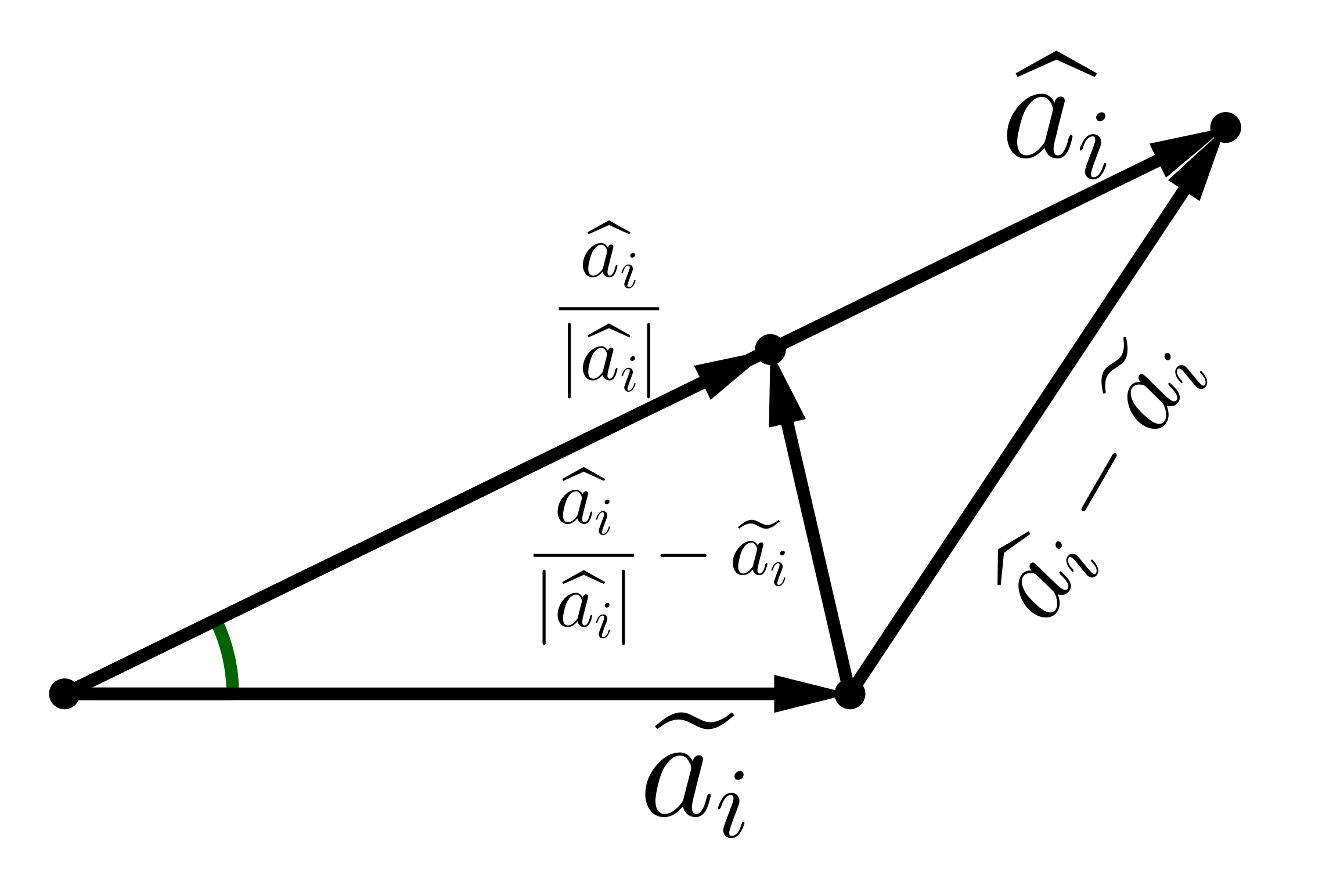}  }
\caption{ }
\label{a_ia_i}
\centering{\includegraphics[width=90mm]{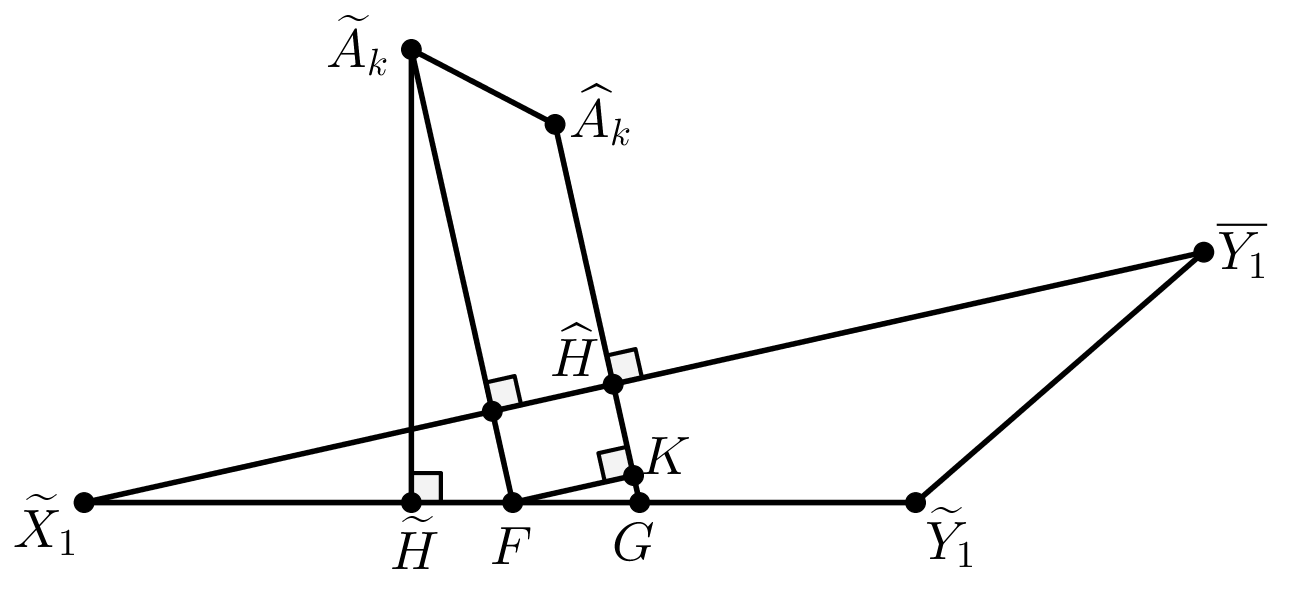} }
\caption{ }
\label{estim_dist}
\end{multicols}
\end{figure}

Now estimate the  length of the vector $ \widehat{a_i} - \widetilde{a_i} $ (see Figure \ref{a_ia_i}).  
The following inequality holds
 \begin{equation}\label{a_i-a_i_1}
| \widehat{a_i} - \widetilde{a_i} | \le 
\left|  \frac{\widehat{a_i}}{|\widehat{a_i}|} - \widetilde{a_i}\right|+ 
\left| \widehat{a_i} - \frac{\widehat{a_i}}{|\widehat{a_i}|} \right|.
 \end{equation} 
Since $\widetilde{a_i}$ is a unite vector, then
\begin{equation}
\left|  \frac{\widehat{a_i}}{|\widehat{a_i}|} - \widetilde{a_i}\right| \le \angle(\widehat{a_i}, \widetilde {a_i} ) 
 \;\;\; \textnormal{и}  \;\;\;
\left|\widehat{a_i} - \frac{\widehat{a_i}}{|\widehat{a_i}|}\right| \le \widehat{ l}_i -1.   
\end{equation} 
From the inequality (\ref{bar_l-1}) we get
  \begin{equation}\label{vspom_a_i_1}
\left|\widehat{a_i} - \frac{\widehat{a_i}}{|\widehat{a_i}|}\right|  < 
\frac {  \cos \frac{\pi}{12}  \left( 4+  \pi^2  (2i+1)^2  \right) }
{  \left( 1-\frac{2}{\pi}a(i+1) \right)^2 } \cdot \varepsilon.  
  \end{equation}
 Using the estimation  (\ref{a_p+q}), we obtain
 \begin{equation}\label{vspom_a_i}
\left|\widehat{a_i} - \frac{\widehat{a_i}}{|\widehat{a_i}|}\right| <
\frac { \cos \frac{\pi}{12} (p+q)^2  \left( 4+  \pi^2  (2i+1)^2  \right) }{ \left( p+q-i-1 \right)^2 } \cdot \varepsilon.  
 \end{equation}
Therefore, from (\ref{a_i-a_i_1}), (\ref{angle_b_ia_i}) and (\ref{vspom_a_i}) we get 
  \begin{equation}\label{a_i-a_i}
| \widehat{a_i} - \widetilde{a_i} | \le \left( c_l (i) + \sum_{j=0}^i  c_\alpha (j) \right) \varepsilon,
 \end{equation}
where
   \begin{equation}\label{c_l}
 c_l (i) =\frac { \cos \frac{\pi}{12} (p+q)^2  \left( 4+  \pi^2  (2i+1)^2  \right) }{ \left( p+q-i-1 \right)^2 },
 \end{equation}
    \begin{equation}\label{c_alpha}
 c_\alpha (j)  =4   \left(   8 \pi     (p+q)^2 \cos \frac{\pi}{12}  \tan^2 \frac{\pi j}{2 (p+q) }  + 1 \right).
 \end{equation}
 Using  (\ref{a_i-a_i}) we estimate the length of $\widehat{ Y}_1 \widetilde{ Y}_1 $ 
 \begin{equation}\label{Y_1Y_1} 
 | \widehat{ Y}_1\widetilde{Y}_1| < 
  \sum_{i=0}^{s+1}  |\widehat{a_i} - \tilde {a_i} | <
 \sum_{i=0}^{s+1}   \left( c_l (i) + \sum_{j=0}^i  c_\alpha (j) \right) \varepsilon. 
 \end{equation}
From  (\ref{angle_b_ia_i}) it follows that the angle $\angle  \widehat{Y}_1\widehat{X }_1\widetilde{Y}_1$  satisfies
 \begin{equation}\label{angle_Y_1X_1Y_1} 
  \angle \widehat{Y}_1\widehat{ X}_1 \widetilde Y_1 < 
 \sum_{i=0}^{s+1}   c_\alpha (i)  \varepsilon.
 \end{equation}
The distance between the vertices $\widehat{A}_k$ and   $\widetilde{ A}_k$  equals
  \begin{equation}\label{A_kA_k} 
 |\widehat{ A}_k\widetilde{A}_k|<
 \sum_{i=0}^k  \left( c_l (i) + \sum_{j=0}^i  c_\alpha (j) \right) \varepsilon. 
 \end{equation}

We drop a perpendicular $\widehat{ A}_k\widehat{ H}$ from the  vertex $\widehat{ A}_k$ 
into the segment  $\widehat{X}_1\widehat{ Y}_1$.
The length of $\widehat{A}_k\widehat{H}$  equals $\widehat{ h }$.
Then we drop the perpendicular $\widetilde{ A}_k\widetilde{ H}$ into the segment $\widetilde{X}_1\widetilde{ Y}_1$ 
and the length of $\widetilde{ A}_k\widetilde{ H}$ equals $\widetilde{ h}$.

Let the point $F$ on $\widetilde{ X}_1\widetilde{Y}_1$ be such that
the segment $\widetilde{ A}_k F$ intersect $\widehat{ X}_1\widehat{ Y}_1$ at right angle.
Then the length of $\widetilde{ A}_k F$ is not less than $ \widetilde{ h }$.
Let the point $G$ on $\widetilde{X}_1\widetilde{ Y}_1$ lie  at the extension 
of the segment $\widehat{ A}_k\widehat{ H}$, and 
  $FK$ is perpendicular to $\widehat{ H} G$  (see Figure~\ref{estim_dist}).
Then the length of $FK$ is not greater than the length of $ \widehat{A}_k\widetilde{A}_k $, 
 and $\angle KFG = \angle\widehat{ Y}_1\widehat{X}_1 \widetilde{ Y}_1$.
From the triangle $GFK$ we get, that
   \begin{equation}\label{FG_1} 
 | FG |= \frac{|FK|}{\cos \angle \widehat{ Y }_1\widehat{ X}_1 \widetilde{ Y}_1  }. 
 \end{equation}
Applying the inequality  $\cos x > 1-\frac{2}{\pi}x$, where $x < \frac{\pi}{2}$,  to (\ref{FG_1}). We obtain
    \begin{equation}\label{FG_2} 
 | FG | < \frac{|\widehat{ A}_k\widetilde{ A}_k |}{ 1-\frac{2}{\pi}  \angle \widehat{ Y}_1\widehat{ X}_1 \widetilde{ Y }_1 }. 
 \end{equation}
So from  (\ref{angle_Y_1X_1Y_1}),  (\ref{A_kA_k}), and (\ref{FG_2})  it follows
   \begin{equation}\label{FG_3} 
   | FG | <  \frac{  \sum_{i=0}^k  \left( c_l (i) + \sum_{j=0}^i  c_\alpha (j) \right) \varepsilon}
  {1- \sum_{i=0}^s 
   \left(   64 \pi     (p+q)^2 \cos \frac{\pi}{12}  \tan^2 \frac{\pi i}{2 (p+q) }  + \frac{8}{\pi} \right) \varepsilon}. 
  \end{equation}
From (\ref{epsilon_estim_1}) and (\ref{FG_3}) we obtain 
       \begin{equation}\label{FG}
   | FG |<  \frac{   \sum_{i=0}^k  \left( c_l (i) + \sum_{j=0}^i  c_\alpha (j) \right) \varepsilon }
  {1-  \frac{(p+q+2)}{2\pi  \cos \frac{\pi}{12}  (p+q)^2 } -8 \sum_{i=0}^{s+1}  \tan^2 \left( \frac{\pi i}{2 (p+q) } \right)  }. 
  \end{equation}

From the our construction it follows
   \begin{equation}\label{vspom_h} 
 \widetilde{ h} \le \widetilde{A}_kF \le  \widehat{ h} + |\widehat{ H} G |+ |\widehat{ A}_k \widetilde {A}_k| + | FG |;  
 \end{equation}
Note, that $|\widehat{ H} G |<   | \widehat{Y}_1\widetilde{ Y}_1 |  $. 
From Lemma~\ref{evcl_dist_vertex_lem} we know, that the distance $ \widetilde{ h} $
satisfies the inequality 
\begin{equation}\label{ }
\widetilde{ h} >   \frac{ \sqrt{3}  }{  4 \sqrt{p^2+q^2+pq} }.  \notag
\end{equation}
Hence from (\ref{vspom_h}) it follows, that
   \begin{equation}\label{bar_h_1} 
      \widehat{ h }    > \frac{ \sqrt{3}  }{  4 \sqrt{p^2+q^2+pq} } -
 | \widehat{Y}_1\widetilde{ Y}_1 |   - 
|\widehat{ A}_k \widetilde {A}_k| - | FG |.  
 \end{equation}
Applying the estimations (\ref{Y_1Y_1}), (\ref{A_kA_k}), (\ref{FG}) 
and the equality $s=\left[ \frac{p+q}{2} \right]+1 $, we get
   \begin{equation}\label{bar_h_1} 
      \widehat{ h }    > \frac{\sqrt{3}}{4 \sqrt{p^2+q^2+pq}  } - c_0
 \sum_{i=0}^{\left[ \frac{p+q}{2} \right]+2}   \left( c_l (i) + \sum_{j=0}^i  c_\alpha (j) \right) \varepsilon,  
 \end{equation}
where  $c_l (i)$ is from  (\ref{c_l}), and $c_\alpha (j)$ is from (\ref{c_alpha}) and
          \begin{equation}\label{c_0 }
 c_0=   \frac{ 
 3 - \frac{(p+q+2)}{\pi  \cos \frac{\pi}{12}  (p+q)^2 } - 
   16 \sum_{i=0}^{\left[ \frac{p+q}{2} \right]+2}  \tan^2 \left( \frac{\pi i}{2 (p+q) } \right) }
  {1-  \frac{(p+q+2)}{2\pi  \cos \frac{\pi}{12}  (p+q)^2 } -
  8 \sum_{i=0}^{\left[ \frac{p+q}{2} \right]+2}  \tan^2 \left( \frac{\pi i}{2 (p+q) } \right)  }, \notag
 \end{equation} 
 From the inequality (\ref{bar_h_1}) we obtain, that if
      $\varepsilon$ satisfies the condition
        \begin{equation}\label{varepsilon} 
     \varepsilon  <   \frac{\sqrt{3}}{4  c_0\sqrt{p^2+q^2+pq}\;
 \sum_{i=0}^{\left[ \frac{p+q}{2} \right]+2}   \left( c_l (i) + \sum_{j=0}^i  c_\alpha (j) \right)} ,   
 \end{equation} 
      then the distance from the vertices of the polygon   $\widehat{ T}_{pq}$ to  $\widehat{ X}_1\widehat{ Y}_1$ is
      greater than zero.
     
     Since we use the estimation  (\ref{epsilon_estim_1}), we get, that if   
    \begin{equation}\label{varepsilon_fin} 
     \varepsilon  <  \min \left\{
      \frac{\sqrt{3}}{4 c_0\sqrt{p^2+q^2+pq}\;
 \sum_{i=0}^{\left[ \frac{p+q}{2} \right]+2}   \left( c_l (i) + \sum_{j=0}^i  c_\alpha (j) \right)};
    \frac{1}{8  \cos \frac{\pi}{12}(p+q)^2 }
  \right\},   
 \end{equation} 
 then the segment  $\widehat{ X}_1\widehat{ Y}_1$ lies inside the polygon  $\widehat{ T}_{pq}$.
 It follows that the arc $X_1Y_1$  on a sphere is also containing inside the polygon $T_{pq}$.
 The arc $X_1Y_1$ corresponds to a simple closed geodesic  $\gamma$ of type $(p,q)$  
 on a regular tetrahedron with the faces angle $\alpha  =\pi/3 + \varepsilon$ in spherical space.
 From Corollary \ref{uniqueness_cor} we get, that this geodesic is unique, up to the rigid motion of the tetrahedron.
 
 Note, that the geodesic  $\gamma$ is invariant under 
  the rotation of the tetrahedron of  the angle $\pi$ over the line passing through 
 the midpoints of the opposite edges of the tetrahedron.
 The rotation of the tetrahedron of the angle  $2\pi/3 $ or $4\pi/3 $  over the line  
 passing through the tetrahedron's  vertex  and the center of its opposite face changes  $\gamma$ into
 another geodesic of  type $(p, q)$.
 
 The rotations over others lines connected other tetrahedron's vertex with the center of its opposite face
 produce the existing geodesics. 
So if  $\varepsilon$ satisfies the condition (\ref{varepsilon_fin}), then
on a regular tetrahedron with the faces angle   $\alpha  = \pi/3 + \varepsilon$  in a spherical space
there exist three different simple closed geodesics of type  $(p, q)$.
Theorem 2 is proved. 
    \end{proof}

B.Verkin Institute for Low Temperature Physics and Engineering 
of the National Academy of Sciences of Ukraine,
Kharkiv, 61103, Ukraine

\textsc {\textit{E-mail address:}} aborisenk@gmail.com, suhdaria0109@gmail.com

\end{document}